\documentclass{svjour3}

\usepackage{amsfonts}
\usepackage{amsmath}
\usepackage{graphicx}
\usepackage{geometry}
\usepackage{csquotes}
\usepackage{amssymb}
\usepackage{algorithm}
\usepackage{algorithmicx}
\usepackage{algpseudocode}

\algrenewcommand\algorithmicrequire{\textbf{Input:}}
\algrenewcommand\algorithmicensure{\textbf{Output:}}

\usepackage{color}
\usepackage[hidelinks]{hyperref}
\usepackage{mathrsfs}
\newcommand{\R}{\mathbb{R}} 
\newcommand{\de}{\mathrm{d}}
\newcommand{\T}{\mathcal{T}}

\renewcommand{\P}{\mathbb{P}} 
\renewcommand{\de}{\mathrm{d}} 

\author{Ludovico Bruni Bruno         \and
        Francesco Dell'Accio \and 
        Wolfgang Erb \and 
        Federico Nudo
}

\institute{Ludovico Bruni Bruno \at
              Department of Mathematics \enquote{Tullio Levi-Civita}, University of Padova, Italy \\
              Istituto Nazionale di Alta Matematica, Roma, Italy
              \email{ludovico.brunibruno@unipd.it}         
           \and
            Francesco Dell'Accio \at
             Department of Mathematics and Computer Science, University of Calabria, Rende (CS), Italy\\ 
             Istituto per le Applicazioni del Calcolo 'Mauro Picone', Naples Branch, C.N.R. National Research Council of Italy, Napoli, Italy
             \email{francesco.dellaccio@unical.it} 
         \and 
             Wolfgang Erb \at
              Department of Mathematics \enquote{Tullio Levi-Civita}, University of Padova, Italy \\
              \email{wolfgang.erb@unipd.it}
         \and 
            Federico Nudo \at
              Department of Mathematics \enquote{Tullio Levi-Civita}, University of Padova, Italy \\
              \email{federico.nudo@unipd.it}
}

\begin{document}


\title{Bivariate polynomial histopolation techniques on Padua, Fekete and Leja triangles}







\date{}

\maketitle

\begin{abstract}
This paper explores the reconstruction of a real-valued function $f$ defined over a domain $\Omega \subset \mathbb{R}^2$ using bivariate polynomials that satisfy triangular histopolation conditions. More precisely, we assume that only the averages of $f$ over a given triangulation $\mathcal{T}_N$ of $\Omega$ are available and seek a bivariate polynomial that approximates $f$ using a histopolation approach, potentially flanked by an additional regression technique. This methodology relies on the selection of a subset of triangles $\mathcal{T}_M \subset \mathcal{T}_N$ for histopolation, ensuring both the solvability and the well-conditioning of the problem. The remaining triangles can potentially be used to enhance the accuracy of the polynomial approximation through a simultaneous regression. We will introduce histopolation and combined histopolation-regression methods using the Padua points, discrete Leja sequences, and approximate Fekete nodes. The proposed algorithms are implemented and evaluated through numerical experiments that demonstrate their effectiveness in function approximation.
\end{abstract}

\section{Introduction}

Let $f$ be a real-valued function defined on a polygonal domain $\Omega \subset \R^2$. We assume that this domain is partitioned into a set of triangles $ \mathcal{T}_N=\left\{t_1, \ldots, t_N \right\} $ 
such that
\begin{equation*} \Omega = \bigcup_{i=1}^N t_i, \quad \text{ and } \quad   \left\lvert t_i \cap t_j \right\rvert = 0, \quad i\neq j. 
\end{equation*}
The main goal of this paper is to explore approximations of the function $f$ through a bivariate polynomial $p(x,y)$, assuming that only the averages
\begin{equation} \label{eq:mui} \mu_i(f) = \frac{1}{|t_i|}\int_{t_i} f(x,y) \de x \de y , \quad i \in \{ 1, \ldots, N\},
\end{equation}
of $f$ over the triangles $t_i$ are available. Such problems arise in various fields, including histosplines~\cite{Schoenberg}, optimal transport~\cite{OptimalTransport}, and mimetic methods~\cite{BeiraoMimetic}. 
The first obstacle in the resolution of such a problem emerges when the number $N$ of data values does not match the dimension $ M \doteq \operatorname{dim} \left(\P_m\left(\mathbb{R}^2\right)\right)$ of the chosen polynomial space, making exact histopolation impossible. 
Even if $N = M$ is satisfied, the 
\emph{histopolation} problem which can be formulated as (cf. \cite{Robidoux})
\begin{equation} \label{eq:histopol} \text{Given a function $f$, find a polynomial $p$ such that $\mu_i(p) = \mu_i(f)$, $i \in \{ 1,\dots,M\}$,} \end{equation}
might not have a unique solution in the space $\P_m \left(\mathbb{R}^2\right)$ of bivariate polynomials of total degree less or equal to $m$. Furthermore, even if a unique solution exists, the histopolation problem \eqref{eq:histopol} may be ill-conditioned~\cite{RothThesis}. Such an ill-conditioning, which is typically due to a disadvantageous selection of triangles, causes inaccuracies in the polynomial approximation of the function $f$ and leads to Runge-like phenomena as, for instance, displayed in Figure~\ref{fig:Runge}. 
\begin{figure}[ht]
	\centering
	\includegraphics[width=0.49\textwidth]{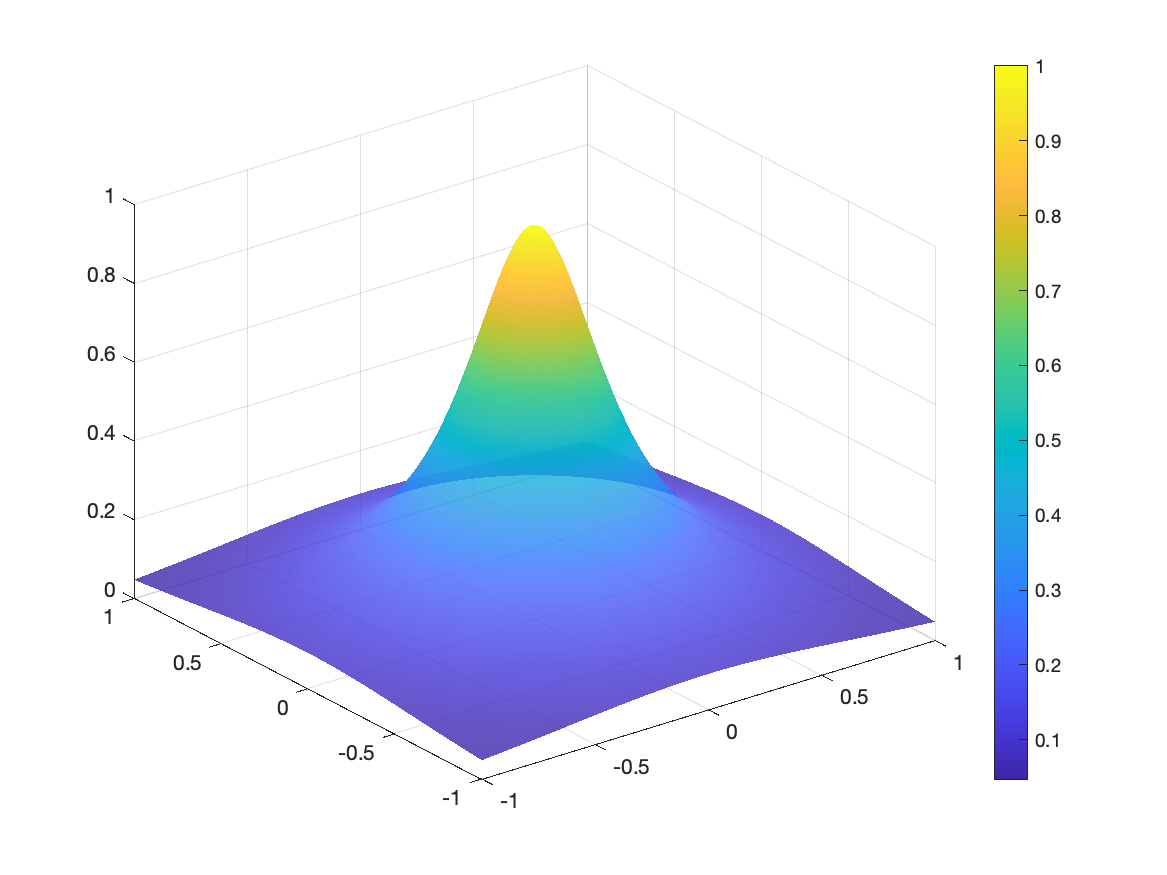}
	\includegraphics[width=0.49\textwidth]{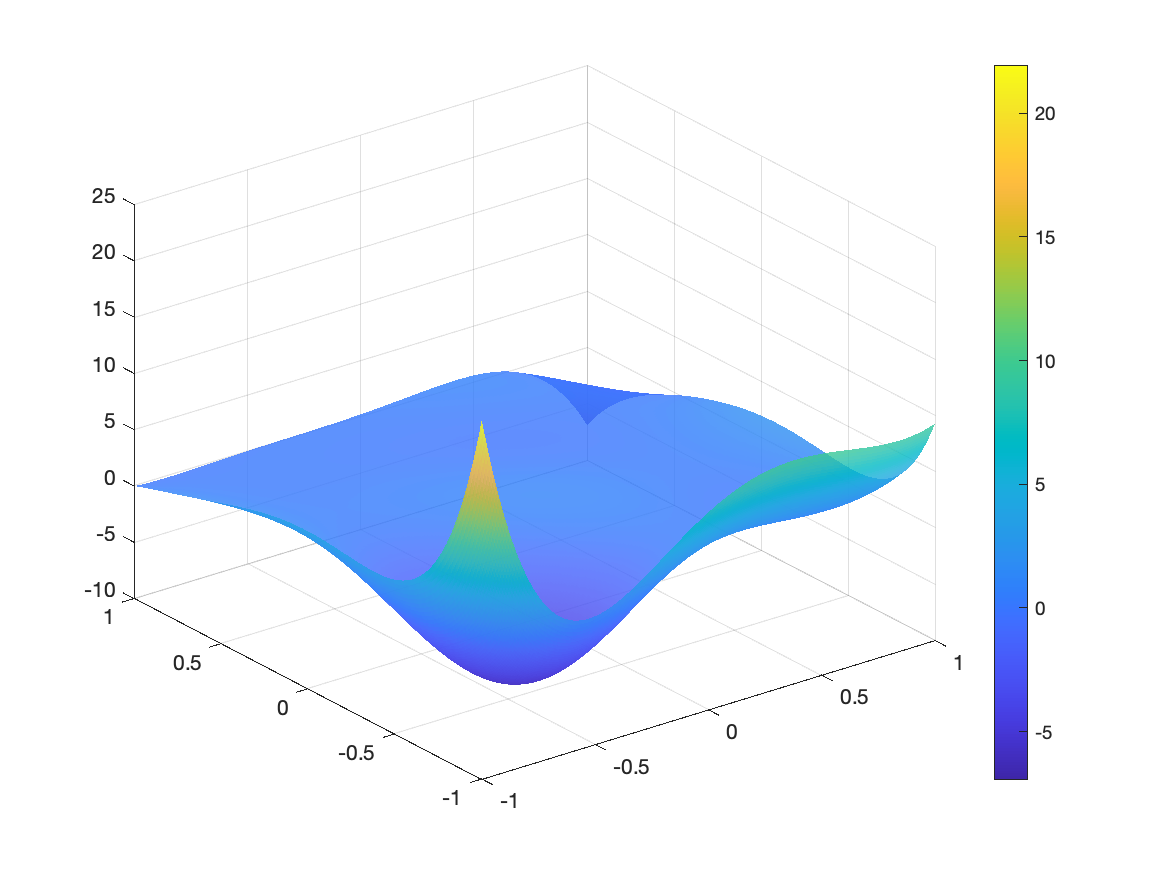}
	\caption{Ill-conditioning of histopolation: the Runge function $ f(x,y) = \frac{1}{1+10(x^2 + y^2)} $ (left) and its polynomial reconstruction from average data on a random selection of triangles (right) in $[-1,1]^2$, using a total polynomial degree $ m = 5 $.}
	\label{fig:Runge}
\end{figure}

To overcome these issues, we refine the triangulation $\T_N$ and sub-sample $\mathcal{T}_N$ on a smaller set
$\T_M \subset \T_N$ of $M < N$ triangles, verifying that the histopolation problem is solvable, i.e., that the collection $ \T_M $ is unisolvent for the polynomial space $\P_m\left(\mathbb{R}^2\right)$, and well-conditioned. Furthermore, additional data values on the remaining triangles may be used to enhance the approximation accuracy through a simultaneous regression. 
This leads to the combined \emph{histopolation-regression} problem 
\begin{equation}\label{consleastsquares} \text{Find a polynomial $p$ such that } \begin{cases} \mu_i(p) = \mu_i(f), & i \in \{ 1, \ldots, M\}, \\ \mu_{M+j}(p) \approx \mu_{M+j}(f), & j \in \{ 1, \dots, N-M\}, \end{cases} \end{equation}
where $p$ lies in a larger polynomial space $\P_d\left(\mathbb{R}^2\right)$, $d > m$, and the set $\T_N$ is reordered so that its first elements are those of $\T_M$. 
In this paper, we will explore three different strategies for histopolation and combined histopolation-regression. The first one is based on the \emph{Padua points}. These set of nodes are well-established in bivariate interpolation \cite{Bos;2006:BLI,bos2007bivariate} and exploit the \emph{Dubiner metric} \cite{Dubiner}
$$ \mathrm{dist} (\xi, \eta) \doteq \max \left\{ \left|\arccos\left(\xi_x\right) - \arccos\left(\eta_x\right) \right|, \left| \arccos\left(\xi_y\right) - \arccos\left(\eta_y\right) \right| \right\} $$
for points $ \xi = \left(\xi_x, \xi_y\right) $ and $ \eta = \left(\eta_x,\eta_y\right) $ in $ \Omega = [-1,1]^2 $. It is in fact conjectured in~\cite{CDMV} that equidistributed nodes with respect to such a distance are near-optimal for bivariate interpolation. The  relation between interpolation on the Padua points and histopolation on the Padua triangles is established through the mean value theorem and the results of Proposition \ref{prop:stability}. Indeed, while the mean value theorem allows to convert the histopolation problem into an interpolation problem which can be interpreted as a perturbation of the interpolation on the Padua points, Proposition \ref{prop:stability} ensures that the solution of the histopolation problem on the Padua triangles converges to the interpolant on the Padua points if the triangulation gets increasingly refined in a uniform sense. For consistency with the definition of the Padua points \cite{PaduaComputational}, we will generally consider histopolation problems on the domain $ \Omega = [-1,1]^2 $. Moreover, we will compare the histopolation results on the Padua triangles with two additional numerical triangle selection strategies that are inspired by established techniques for generating discrete Leja sequences and approximate Fekete points in classical polynomial interpolation. 


\noindent \textbf{Outline of the paper}. The paper is organized as follows. In Section~\ref{sec2}, we extend the concept of polynomial interpolation on Padua points to bivariate polynomial histopolation on triangles. For this, we will present a histopolation procedure over a triangulation of the domain $[-1,1]^2$, together with an enhancement strategy based on an additional regression. We will establish bounds for the respective Lebesgue constants, i.e., the norms of the histopolation operator and the histopolation-regression operator. In Section~\ref{sec3new}, we will explore two additional histopolation-regression techniques for bivariate domains in which the numerical extraction of the triangles follows the strategy for the calculation of the approximate Fekete points and Leja sequences. These additional procedures will be used for the comparisons in Section~\ref{sec4}, where we will provide several numerical results that demonstrate the accuracy of the histopolation method based on the Padua triangles.

\section{Histopolation-regression techniques based on Padua triangles}\label{sec2}

\subsection{Determination of Padua triangles}
For the domain $ \Omega = [-1,1]^2 $, the set of Padua points $X_M^{\mathrm{Pad}}$ of size $ M = (m+1)(m+2)/2 $, $m \in \mathbb{N}$, is defined as
\begin{equation}\label{PaduaPoints}
	X_M^{\mathrm{Pad}}=\left\{\left((-1)^{i+j}\cos\left(\frac{\pi j}{m+1}\right),(-1)^{i+j}\cos\left(\frac{i\pi}{m}\right)\right)\, : \, 0\le i+j\le m\right\} .
\end{equation}
Given a triangulation $ \mathcal{T}_N = \left\{ t_i\right\}_{i=1}^N $ of size $N \geq M$, we define a corresponding set of Padua triangles as 
\begin{equation*}\label{PaduaTriangles}
	\mathcal{T}^{\mathrm{Pad}}_M \doteq \left\{ t \in \mathcal{T}_N \, :\,  \boldsymbol{x}^{\mathrm{Pad}} \in t \text{ for some }\boldsymbol{x}^{\mathrm{Pad}} \in 	X_M^{\mathrm{Pad}}  \right\}.
\end{equation*}

Since the set of nodes $ X_M^{\mathrm{Pad}} $ is unisolvent for the space of polynomials $ \P_m\left(\mathbb{R}^2\right)$, see~\cite{Bos;2006:BLI,bos2007bivariate}, we are interested in obtaining an attribution map  
\begin{equation} \label{eq:PPtoPT}
\varphi: \quad X^{\mathrm{Pad}}_M \ni \boldsymbol{x}^{\mathrm{Pad}} \mapsto t_i \in \mathcal{T}^{\mathrm{Pad}}_M
\end{equation}
which is not only well-defined but also injective, such that the number of Padua triangles in $\mathcal{T}^{\mathrm{Pad}}_M$ corresponds again to the dimension $M = (m+1)(m+2)/2$ of the space $\mathbb{P}_m\left(\mathbb{R}^2\right)$ of bivariate polynomials of total degree $m$.  For any point $ \boldsymbol{x} \in \R^2 $ and triangle $ t_i$ with vertices $\boldsymbol{v}^{(i)}_1, \boldsymbol{v}^{(i)}_2, \boldsymbol{v}^{(i)}_3,$ one has that $ \boldsymbol{x} \in t_i$ if and only if the barycentric coordinates of the triangle $t_i$ at the points $\boldsymbol{x}$ are nonnegative, that is
\begin{eqnarray*}
\lambda_{1}^{(i)}({\boldsymbol{
x}})&:=&\frac{\left(\boldsymbol{x}-\boldsymbol{v}^{(i)}_3\right)\times\left(\boldsymbol{v}^{(i)}_2-\boldsymbol{v}^{(i)}_3\right)}{\left(\boldsymbol{v}^{(i)}_1-\boldsymbol{v}^{(i)}_3\right)\times\left(\boldsymbol{v}^{(i)}_2-\boldsymbol{v}^{(i)}_3\right)}\ge0, 
\\ 
\lambda_{2}^{(i)}({\boldsymbol{
x}})&:=&\frac{\left(\boldsymbol{x}-\boldsymbol{v}^{(i)}_3\right)\times\left(\boldsymbol{v}^{(i)}_3-\boldsymbol{v}^{(i)}_1\right)}{\left(\boldsymbol{v}^{(i)}_1-\boldsymbol{v}^{(i)}_3\right)\times\left(\boldsymbol{v}^{(i)}_2-\boldsymbol{v}^{(i)}_3\right)}\ge0, \\
\lambda_{3}^{(i)}({\boldsymbol{
x}})&:=&\frac{\left(\boldsymbol{x}-\boldsymbol{v}^{(i)}_1\right)\times\left(\boldsymbol{v}^{(i)}_1-\boldsymbol{v}^{(i)}_2\right)}{\left(\boldsymbol{v}^{(i)}_1-\boldsymbol{v}^{(i)}_3\right)\times\left(\boldsymbol{v}^{(i)}_2-\boldsymbol{v}^{(i)}_3\right)}\ge0.
\end{eqnarray*}
Checking these conditions is computationally expensive. To be more cost-efficient, we can apply a \emph{point-in-triangle} algorithm \cite{HormannAgathosPIP} for determining whether a Padua point $\boldsymbol{x}^{\mathrm{Pad}}$ belongs to some $ t_i \in \mathcal{T}_N$. For this, consider a point $ \boldsymbol{x} \in \Omega $ and a triangle $ t_i $ of vertices $ \boldsymbol{v}^{(i)}_1, \boldsymbol{v}^{(i)}_2, \boldsymbol{v}^{(i)}_3 $. Then $ \boldsymbol{x} $ defines three (possibly degenerate) triangles by replacing a vertex of $ t_i $ by $ \boldsymbol{x} $. Denote by $ t_i^j $ the triangle obtained from $ t_i $ by replacing $ \boldsymbol{v}^{(i)}_j $ with $ \boldsymbol{x} $. Then, $ \boldsymbol{x} \in t_i $ if and only if $ \left| t_i \right| = \left| t_i^1 \right| + \left| t_i^2 \right| + \left| t_i^3 \right| $. Notice that this only requires the computation of vector products, or equivalently determinants, in place of the resolution of a linear system.

This condition alone does not guarantee that the mapping $  \boldsymbol{x} \mapsto t_i $ is well-defined. Indeed, suppose that $ \boldsymbol{x} $ belongs to an edge which is common to two triangles. Then, the admissibility condition holds for both triangles, and the association $ \boldsymbol{x} \mapsto t_i $ does not constitute a function. The  situation gets worse if $ \boldsymbol{x} = \boldsymbol{v}_j^{(i)} $ for some $i$ and $ j \in \{1,2,3\}$; then $ \boldsymbol{x} $ may belong to an even larger number of triangles in $ \mathcal{T}_N$. In these situations, we resolve possible ambiguities by ordering the triangles in $ \mathcal{T}_N$ and by linking $ \boldsymbol{x} $ with the first admissible triangle in the list. 

\begin{remark}
In numerical experiments we could observe that if the triangulation $\T_N$ bears some regularity, for instance, by considering triangles with similar area and avoiding that too many triangles share the same vertex, also different attribution rules $ \boldsymbol{x}^{\mathrm{Pad}} \mapsto t_i $ yield comparable results in terms of the properties of the approximation polynomial.  
\end{remark}

The injectivity of the mapping \eqref{eq:PPtoPT} can be related to the maximal edge length of the triangles in the triangulation $\mathcal{T}_N$. Indeed, let
$$ h_{\max} \doteq \max_{t_i \in \mathcal{T}_N} \max_{\boldsymbol{v}_j^{(i)}, \boldsymbol{v}_k^{(i)} \in t_i} \left\lVert \boldsymbol{v}_j^{(i)} - \boldsymbol{v}_k^{(i)} \right\rVert_2 $$
be the maximum edge length in the triangulation $ \mathcal{T}_N$. If this length is less than the distance between any two Padua points, the map $ \varphi $ is injective. The explicit representation \eqref{PaduaPoints} of the Padua points makes it possible to link the polynomial degree $ m $ with the length $ h_{\max} $.

\begin{proposition}\label{thmimp}
Let $\mathcal{T}_N$ be a triangulation of $\Omega=[-1,1]^2$. If 
\begin{equation}\label{gencond}    m<\frac{\pi}{\arccos\left(1-\sqrt{\frac{h^2_{\max}}{2}}\right)}-1,
    \end{equation}
    each Padua triangle in $\mathcal{T}_M^{\mathrm{Pad}}$ contains exactly one Padua point. Equivalently, the map $\varphi$ in \eqref{eq:PPtoPT} is injective.
\end{proposition}
\begin{proof}
Let $m \in \mathbb{N}$, and consider the set of Padua points $X_M^{\mathrm{Pad}}$ defined in \eqref{PaduaPoints}. To obtain injectivity of the map $\varphi$, we must guarantee that the distance between two Padua points is less than $h_{\max}$. Observing that the Padua points accumulate in the corners of $ \Omega $ (see also \cite{Bos;2006:BLI}), we get the inequalities
\begin{eqnarray*}
\min_{\substack{\boldsymbol{x}, \boldsymbol{y} \in X_M^{\mathrm{Pad}} \\ \boldsymbol{x} \ne \boldsymbol{y} }} \left\lVert \boldsymbol{x} - \boldsymbol{y} \right\rVert_2^2 &\ge& \left( 1 - \cos\left( \frac{\pi}{m+1} \right) \right)^2+\left( 1 - \cos\left( \frac{\pi}{m} \right) \right)^2\\ &>&2\left( 1 - \cos\left( \frac{\pi}{m+1} \right) \right)^2 > h_{\max}^2 .
\end{eqnarray*}
By solving the last inequality for $ m $, we obtain the result.
\end{proof}

As a computational consequence of this result, if the triangulation $ \mathcal{T}_N $ satisfies the hypotheses of Proposition \ref{thmimp}, the following Algorithm \ref{algPad} returns a set of Padua triangles $ \mathcal{T}^{\mathrm{Pad}}_M $ with the correct size $M$ corresponding to the dimension of $\P_m(\mathbb{R}^2)$.

\begin{algorithm}[h!]
	\caption{Extraction of Padua triangles}
	\begin{algorithmic}[1]
		\Require a triangulation $\mathcal{T}_N =\left\{ t_1,\dots,t_N \right\} $ of $ \Omega $, the Padua points $ X_M^{\mathrm{Pad}} $	
		\For{$ \boldsymbol{x}^{\mathrm{Pad}} \in X_M^{\mathrm{Pad}} $}
			\For{ $ i \in \{1, \ldots, N\} $}
                \State check if $\boldsymbol{x}^{\mathrm{Pad}} \in t_i$ by the point-in-triangle algorithm
				\If {$ \boldsymbol{x}^{\mathrm{Pad}} \in t_i $}
				\State add $ t_i $ to $\mathcal{T}^{\mathrm{Pad}}_M$
				\State break inner loop
				\EndIf
			\EndFor
			\EndFor
    \Ensure the set of Padua triangles $\mathcal{T}^{\mathrm{Pad}}_M$
	\end{algorithmic}
	\label{algPad}
\end{algorithm}

\subsection{Regular Friedrichs–Keller triangulations}
Of course, the maximal distance $ h_{\max} $ depends on the chosen triangulation $ \mathcal{T}_N$. For particular regular triangulations, the maximal distance $ h_{\max} $ can be estimated explicitly. As an example, we consider a regular Delaunay triangulation $\mathcal{T}^{\mathrm{reg}}_{N}$ constructed upon the quadratic grid of vertices 
$$ \left( \frac{2 i }{n} - 1, \frac{2 j}{n} - 1 \right) , \quad i,j \in \{0, \ldots, n\}.$$
Specifically, the triangulation $\mathcal{T}^{\mathrm{reg}}_{N}$ is obtained by connecting each interior node of the quadratic grid to six of its eight neighbors, as illustrated in Figure \ref{fig:PadTriangle}. It forms a simple Delaunay triangulation \cite[Chapter 1]{Edelsbrunner} which is also known as Friedrichs–Keller triangulation \cite[p. 64]{Knabner}. This triangulation finds large application in numerical analysis, particularly in finite element methods \cite{Rahla}. It is easy to check that $ \mathcal{T}^{\mathrm{reg}}_{N} $ contains $ N = 2n^2 $ triangles. In Figure~\ref{fig:PadTriangle}, along with the triangulation $ \mathcal{T}^{\mathrm{reg}}_{N}$, also the Padua points $X_M^{\mathrm{Pad}} $ and the triangles $ \mathcal{T}_M^{\mathrm{Pad}} $ extracted by Algorithm \ref{algPad} are shown.  

\begin{figure}[ht]
    \centering
    \includegraphics[width=0.49\textwidth]{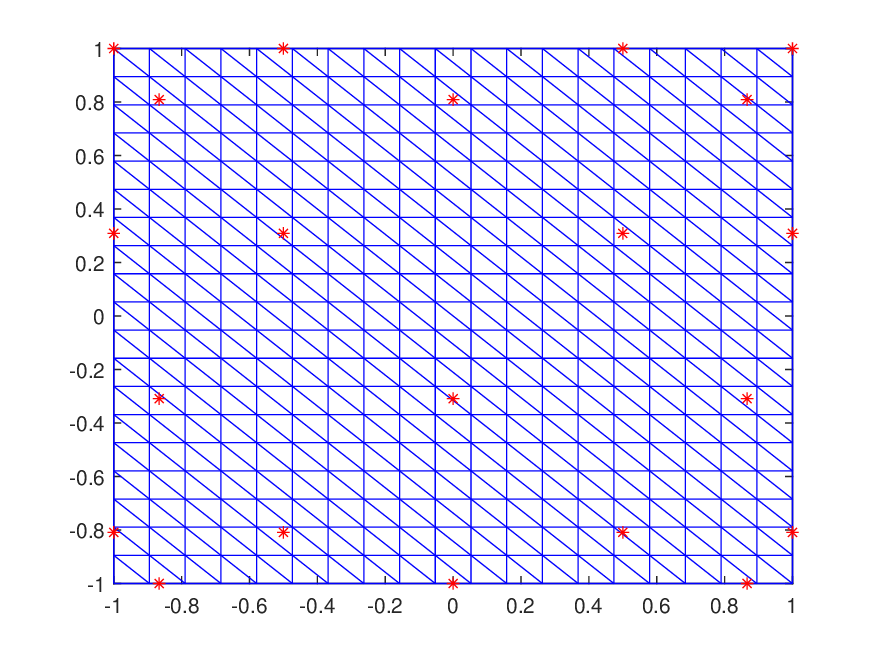}
        \includegraphics[width=0.49\textwidth]{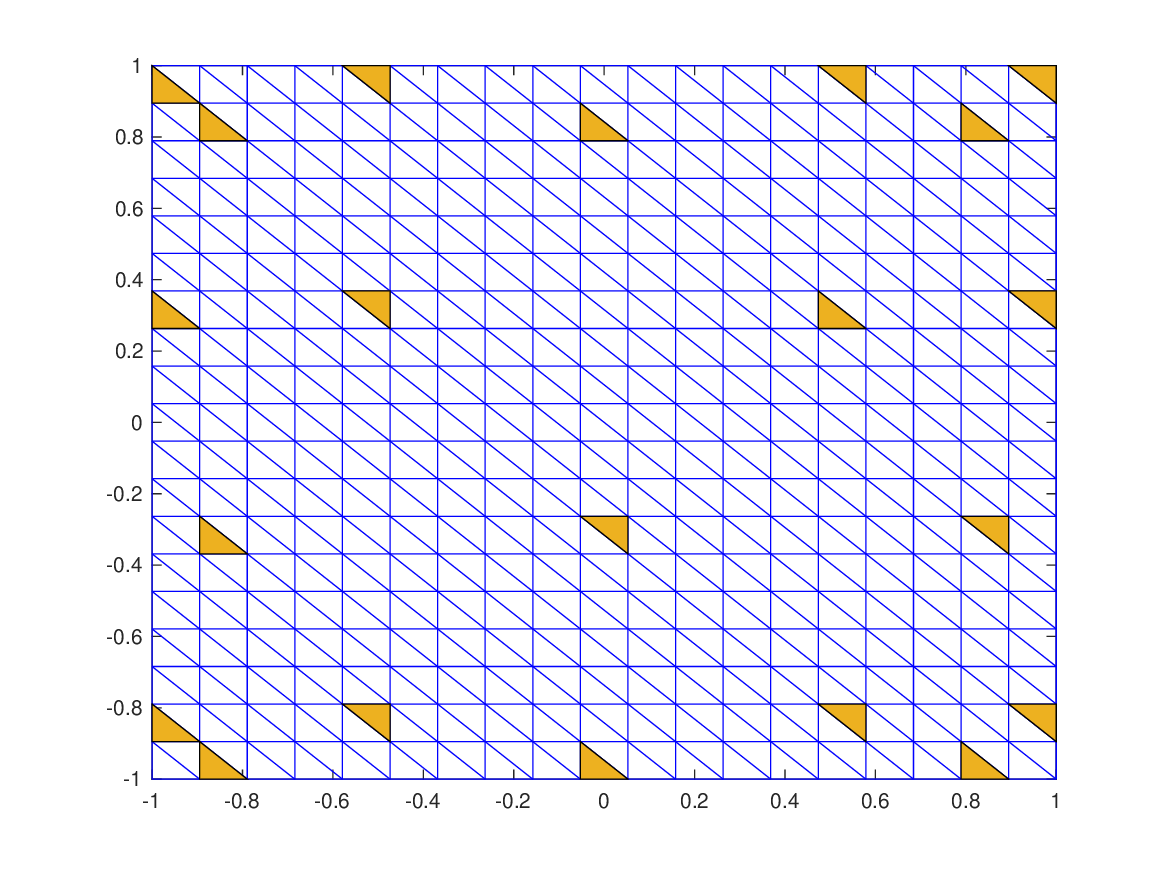}
    \caption{Padua points (\textcolor{red}{*}) and the corresponding Padua triangles for $m = 6$ in $\mathcal{T}^{\mathrm{reg}}_{800}$ (i.e. with $n = 20$). }
    \label{fig:PadTriangle}
\end{figure}
%
%
%

\vspace{2mm}

The particular construction of the Friedrichs-Keller triangulation allows to determine an exact value for $ h_{\max} $ as a function of the step size $2/n$, and hence to estimate the required number of triangles $N$ to get the injectivity of the selection map, making it possible to apply Eq. \eqref{gencond} directly to this particular triangulation.

\begin{corollary}
Let $\Omega=[-1,1]^2$ be partitioned by the regular Friedrichs–Keller triangulation $ \mathcal{T}^{\mathrm{reg}}_{N} $ with $N = 2n^2$. Then,
\begin{equation*}
    h_{\max} = \frac{2 \sqrt{2}}{n} = \frac{4}{\sqrt{N}},
\end{equation*}
and the condition~\eqref{gencond} of Proposition~\ref{thmimp} can be simplified to
\begin{equation}\label{conditionOnm}
    m<\frac{\pi}{\arccos\left(\frac{n-2}{n}\right)}-1.
\end{equation}
\end{corollary}

\subsection{Histopolation on the Padua triangles}

The polynomial interpolation problem in $\P_m\left(\mathbb{R}^2\right)$ based on the Padua nodes $ X_M^{\mathrm{Pad}} $ is uniquely solvable and well-conditioned. In a more formalized way, the Padua points form a \emph{norming set} \cite{NormingSet} for the space $ \P_m\left(\mathbb{R}^2\right)$ with a norming constant $ \mathscr{L}_m \left(X_M^{\mathrm{Pad}}\right)$ that coincides with the (nodal) Lebesgue constant of the polynomial interpolation on the Padua points. Further, unisolvence of norming sets is \emph{stable}, in the sense that a slight perturbation of the nodes $ X_M^{\mathrm{Pad}} $ (or more generally also for sets of averaging functionals \cite{MockSegments}) is still unisolvent. The maximal possible perturbation is estimated in \cite[Proposition 1]{PiazzonVianello} and, for $ \Omega = [-1, 1]^2$, it amounts to a perturbation of up to $ \varepsilon = \frac{\alpha}{2 m^2 \mathscr{L}_m \left(X_M^{\mathrm{Pad}}\right)} $ for any $ \alpha \in [0,1) $.

If the triangulation $ \mathcal{T}_N$ satisfies the hypothesis of Proposition \ref{thmimp}, the cardinality of the Padua triangles $ \mathcal{T}_M^{\mathrm{Pad}} $ coincides with $ M = \dim \left(\P_m\left(\mathbb{R}^2\right)\right)$. The following result gives a necessary condition on $ h_{\max} $ in order to guarantee the unisolvence of the set $ \mathcal{T}_M^{\mathrm{Pad}} $ for histopolation -- or equivalently, the well-posedness of the histopolation operator
\begin{equation} \label{ea:histopolationPadua}
\Pi_{m}^{\mathrm{Pad}} : C (\Omega) \to \P_m\left(\mathbb{R}^2\right)
\end{equation}
that maps a function $f$ to a polynomial $\Pi_{m}^{\mathrm{Pad}}[f]$ of degree $m$ satisfying the histopolation conditions \eqref{eq:histopol} on the Padua triangles $\T_M^{\mathrm{Pad}}$.

\begin{proposition} \label{prop:stability}
	Let $ \mathcal{T}_N $ be a triangulation of $ \Omega $. If for some $0 < \alpha < 1$ we have
	\begin{equation} \label{eq:condhmax}
		h_{\max} < \min \left\{\frac{\alpha}{2 K \left(m \ln m\right)^2}, \sqrt{2} \left( 1 - \cos \left( \frac{\pi}{m+1} \right) \right) \right\}, 
		\end{equation}
where $K$ is linked to the Lebesgue constant of the Padua points, then
	the set $ \mathcal{T}_M^{\mathrm{Pad}} $ is unisolvent in $ \P_m \left(\mathbb{R}^2\right) $ and the histopolation operator $\Pi_{m}^{\mathrm{Pad}}$ well-defined.
\end{proposition}

\begin{proof}
	If the condition \eqref{eq:condhmax} is satisfied, in particular if $ h_{\max} < \sqrt{2} \left( 1 - \cos \left( \frac{\pi}{m+1} \right) \right) $, the attribution map $ X_M^{\mathrm{Pad}} \to \mathcal{T}_M^{\mathrm{Pad}} $ is a bijection by Proposition \ref{thmimp}. Since the mean value theorem \cite[p. 126]{Zorich} states that for a continuous function $f$ we get $ \frac{1}{|t|} \int_{t} f(x,y) \de x \de y = f(\boldsymbol{p}) $ for some point $ \boldsymbol{p} \in t $ (more precisely, in its interior), the histopolation condition on some triangle $ t \in \mathcal{T}_M^{\mathrm{Pad}}$ is equivalent to an interpolation condition on some unknown node $\boldsymbol{p}  \in t$ that depends on $f$. We consider then $ \delta = \left\Vert \boldsymbol{x}^{\mathrm{Pad}} - \boldsymbol{p} \right\Vert_2 $, where $\boldsymbol{x}^{\mathrm{Pad}}$ denotes the Padua point lying in the triangle $t$. Now, since $$ \delta < h_{\max} < \frac{\alpha}{2 K \left(m \ln m\right)^2} \leq \frac{\alpha}{2 m^2 \mathscr{L}_m \left(X_M^{\mathrm{Pad}}\right)},$$ the prerequisites of \cite[Propostion 1]{PiazzonVianello} on the distance between $\boldsymbol{x}^{\mathrm{Pad}}$ and its perturbation $\boldsymbol{p}$ are satisfied; we also used here the fact that the Lebesgue constant for the Padua points satisfies $\mathscr{L}_m \left(X_M^{\mathrm{Pad}}\right) \leq K \left(\ln m\right)^2$ for some constant $K > 0$. Since this bound holds for any triangle $ t \in \mathcal{T}_M^{\mathrm{Pad}}$ and any function $f \in C(\Omega)$, \cite[Propostion 1]{PiazzonVianello} guarantees that the averaging functionals on the Padua triangles $ \mathcal{T}_M^{\mathrm{Pad}} $ provide a norming set for $ \P_m\left(\mathbb{R}^2\right) $ and, hence, the set $ \mathcal{T}_M^{\mathrm{Pad}}$ is unisolvent for histopolation in $ \P_m \left(\mathbb{R}^2\right) $. The estimate $\mathscr{L}_m \left(X_M^{\mathrm{Pad}}\right) \leq K \left(\ln m\right)^2 $ for the Lebesgue constant of the Padua points is, for instance, given in \cite{Bos;2006:BLI}. 
\end{proof}

\subsection{Histopolation-regression based on the Padua triangles} \label{sect:regressionmethod}

It may happen that the data values $\mu_i(f)$ originate from an already collected set of samplings that contains also average values over non-Padua triangles, i.e., averages of $f$ over triangles in the set $ \mathcal{T}_N \setminus \mathcal{T}_M^{\mathrm{Pad}} $. In place of discarding such samplings, one may perform an additional regression on such data to improve the quality of the approximant. For this reason, we formalize the histopolation-regression problem described in \eqref{consleastsquares} mathematically as a constrained least-squares problem~\cite{DellAccio:2022:GOT,dell2024extension,de2024mixed}.  


To this end, we consider an integer $ d > m $ such that 
\begin{equation*}\label{DeM}
D-M \doteq \dim\left(\P_d\left(\mathbb{R}^2\right)\right) - \dim \left(\P_m\left(\mathbb{R}^2\right) \right) \leq \# \mathcal{T}_N - \# \mathcal{T}_M.    
\end{equation*}
We assume that $ \{ p_1, \ldots, p_D \} $ forms a basis for $\P_d \left(\mathbb{R}^2\right) $ and is ordered such that $ \{ p_1, \ldots, p_M \} $ spans $ \P_m \left(\mathbb{R}^2\right) $. 
We also assume that the set $\T_N$ is ordered so that its first $M$ elements are the Padua triangles $\T_M^{\mathrm{Pad}}$. Under these assumptions, we define the matrices 
\begin{eqnarray} \label{eq:rectVanderm}
W&\doteq&\left[W_{i,j}\right]=\left[\mu_i\left(p_j\right)\right]\in\mathbb{R}^{N \times D}, \\ C&\doteq&\left[C_{i,j}\right]=\left[\mu_i\left(p_j\right)\right]\in\mathbb{R}^{M \times D},
\end{eqnarray}
and the data vectors 
\begin{equation*}\label{bandd}
\boldsymbol{b}=\left[b_1,\dots,b_N\right]^T \doteq \left[\mu_1(f),\dots,\mu_{N}(f)\right]^T, \quad \boldsymbol{d}=\left[d_1,\dots,d_M\right]^T \doteq \left[\mu_1(f),\dots,\mu_M (f)\right]^T.
\end{equation*}

Using this terminology, we formulate the histopolation-regression problem~\eqref{consleastsquares} mathematically 
as the constrained least-squares problem
\begin{equation} \label{LSProblem} \min_{\boldsymbol{a}\in \mathbb{R}^{D}}\left\lVert W\boldsymbol{a}-\boldsymbol{b}\right\rVert_2, \quad \text{subject to } \quad C \boldsymbol{a} = \boldsymbol{d},
\end{equation} 
where $\left\lVert \cdot\right\rVert_2$ denotes the discrete $2$-norm. As the first $M$ elements of $\T_N$ correspond to the Padua triangles, the matrix $C$ is of full rank $M$ under the assumptions of Proposition \ref{prop:stability}. Further, if we assume that also the matrix $W$ is of full rank, the minimization problem \eqref{LSProblem} has a unique solution. The latter can, for example, be guaranteed if $\T_N$ contains the Padua triangles $\T_D^{\mathrm{Pad}}$ and the respective assumptions of Proposition \ref{prop:stability} are satisfied. 

The solution of the constrained minimization problem~\eqref{LSProblem} can be derived using the method of Lagrange multipliers, as for instance described in \cite{boyd2018introduction}. For this, we require the Lagrangian function
\begin{equation*}
    \mathcal{F}(\boldsymbol{a},\boldsymbol{z})=\left\lVert W\boldsymbol{a}-\boldsymbol{b}\right\rVert_2^2+\sum_{k=1}^Mz_k\left(\boldsymbol{c}_k^T\boldsymbol{a}-d_k\right), \quad \boldsymbol{a}\in\mathbb{R}^D, \, \boldsymbol{z}\in\mathbb{R}^M,
\end{equation*} 
where $\boldsymbol{c}_i^T$ denotes the $i$-th row of the matrix $C$.
If $\hat{\boldsymbol{a}}$ is a solution of~\eqref{LSProblem}, then there exists a vector $\hat{\boldsymbol{z}}=\left[\hat{z}_1,\dots,\hat{z}_{M}\right]^T$ of Lagrange multipliers such that the following conditions hold:
\begin{eqnarray*}
    {\partial_{a_i} \mathcal{F}}\left(\hat{\boldsymbol{a}},\hat{\boldsymbol{z}}\right)&=&0, \quad i=1,\dots,D,\\
    {\partial_{z_j} \mathcal{F}}\left(\hat{\boldsymbol{a}},\hat{\boldsymbol{z}}\right)&=&0,\quad j=1,\dots,M.
\end{eqnarray*}
These $D + M$ conditions can equivalently be expressed in matrix form as
\begin{equation}\label{LagrangeMultmethod}
	\begin{bmatrix}
		2 W^T W & C^T  \\
		C & 0  \\
	\end{bmatrix}
	\begin{bmatrix}
		\hat{\boldsymbol{a}} \\
		\hat{\boldsymbol{z}} \\
	\end{bmatrix}=
	\begin{bmatrix}
		2 W^T\boldsymbol{b} \\
		\boldsymbol{d} \\
	\end{bmatrix}.
\end{equation}

Once this system is solved, the first $ D $ elements of the solution  $\hat{\boldsymbol{a}}=\left[\hat{a}_1,\dots,\hat{a}_D \right]^T$ allow to define the histopolation-regression operator
\begin{align}\label{operatorPhat}
\hat{\Pi}_d^{\mathrm{Pad}} : \quad C (\Omega) & \to \P_d\left(\mathbb{R}^2\right), \\ \nonumber
	f & \mapsto \sum_{i=1}^D \hat{a}_i p_i,
\end{align}
that maps a continuous function to the polynomial approximant $\hat{\Pi}_d^{\mathrm{Pad}}[f]$.

\begin{remark}
   There are several scenarios in which the usage of a histopolation-regression approach \eqref{LSProblem} has advantages compared to a pure histopolation method. For instance, it might happen that the number of triangles $N$ in the triangulation $\mathcal{T}_N$ does not exactly match the dimension $\binom{m+2}{2}$ of the space of polynomials $ \P_m\left(\mathbb{R}^2\right) $. In this case, while the construction of a histopolant requires the selection of an appropriate subset of $\mathcal{T}_N$, the histopolation-regression approach can exploit the entire data on $\mathcal{T}_N$. Furthermore, similar to the nodal case, increasing the number of triangles used for histopolation leads to an increasingly bad-conditioned histopolation problem, yielding larger error propagations in the calculation of the solutions. The histopolation-regression strategy is able to reduce this bad-conditioning if the parameters $m$ and $d$ are appropriately selected.
\end{remark}

%

\begin{remark}
    By construction, the histopolation-regression operator \eqref{operatorPhat} is a linear projector, that is
    $$ \hat{\Pi}_d^{\mathrm{Pad}} \left[\lambda_1 f + \lambda_2 g\right] = \lambda_1 \hat{\Pi}_d^{\mathrm{Pad}}[f] + \lambda_2 \, \hat{\Pi}_d^{\mathrm{Pad}}[g], \quad \forall \ \lambda_1,  \lambda_2 \in \R, \ \ f,g \in C(\Omega), $$
    and 
    \begin{equation}\label{cind2}
     \hat{\Pi}^{\mathrm{Pad}}_{d}\left[\hat{\Pi}^{\mathrm{Pad}}_{d} [f]\right]=\hat{\Pi}^{\mathrm{Pad}}_{d} [f], \qquad f \in C(\Omega).
    \end{equation}
As a consequence, for each $ f \in C (\Omega) $, one has a Lebesgue-type inequality of the form
\begin{equation} \label{eq:lebesguelikeinequality}
	 \left\Vert f -  \hat{\Pi}_d^{\mathrm{Pad}}[f] \right\Vert_\infty \leq \left( 1 + \left\lVert \hat{\Pi}_d^{\mathrm{Pad}} \right\rVert_{\mathrm{op}} \right) \min_{p \in \P_d\left(\mathbb{R}^2\right)} \left\Vert f - p \right\Vert_\infty .
\end{equation}
\end{remark}

\subsection{Bounding the Lebesgue constant for the histopolation-regression method}

A common strategy to quantify the numerical conditioning of an approximation or interpolation operator $ \Pi $, is to consider its operator norm $ \left\Vert \Pi \right\Vert_{\mathrm{op}} $. Such an operator norm is induced by a vector space norm for the corresponding function spaces, and in many cases this norm corresponds to the uniform norm. In several contexts, the quantity $ \left\Vert \Pi \right\Vert_{\mathrm{op}} $ is termed \emph{Lebesgue constant}, and it is readily seen in the Lebesgue inequality \eqref{eq:lebesguelikeinequality} that it quantifies the error between a function and its polynomial approximant in terms of the best polynomial approximation. More detailed explanations on the role of the Lebesgue constant can be found in \cite[Chapter 2]{Davis:1975:IAA}.

To get an accessible bound of the operator norm $ \| \hat{\Pi}_d^{\mathrm{Pad}} \|_{\mathrm{op}} $ in case of the histopolation-regression operator $\hat{\Pi}_d^{\mathrm{Pad}}$, we follow the direct elimination method in \cite[Chapter 5]{bjorck2024numerical}, and have a closer look at the computation of the coefficients $\hat{\boldsymbol{a}}$ in the expansion \eqref{operatorPhat}. The idea consists of splitting $\hat{\boldsymbol{a}} =\left[\hat{\boldsymbol{a}}_1, \hat{\boldsymbol{a}}_2\right]^T $ into two parts, with $ \hat{\boldsymbol{a}}_1 \in \R^M $ and $ \hat{\boldsymbol{a}}_2 \in \R^{D-M} $. Then, we may factorize $ C = QR = Q \left[R_1 , R_2 \right]$, being $ R_1 \in \R^{M \times M} $ a nonsingular, upper triangular matrix. 
Plugging this into \eqref{LagrangeMultmethod}, one immediately retrieves
\begin{equation}\label{a1eq}
\hat{\boldsymbol{a}}_1 = R_{1}^{-1} \left(Q^T\boldsymbol{d} - R_{2} \hat{\boldsymbol{a}}_2\right).    
\end{equation}
Likewise, splitting $ W = \left[W_1 , W_2\right] $ with $ W_1 \in \R^{N \times M} $, one obtains
\begin{equation}\label{a2eq}
	\hat{\boldsymbol{a}}_2=\left(A^{T} A\right)^{-1}A^{T}\boldsymbol{b}_1,
\end{equation}
where we have set $ A \doteq W_2 - W_1 R_{1}^{-1} R_{2} $ and $ \boldsymbol{b}_1 \doteq \boldsymbol{b} - W_1 R_{1}^{-1} Q^T\boldsymbol{d} $. We follow now the strategy given in \cite{Submitted:2025:OTE} to bound $ \| \hat{\Pi}_d^{\mathrm{Pad}} \|_{\mathrm{op}} $ in terms of these factors, isolating and identifying the contributions from $ \hat{\boldsymbol{a}}_1 $ and $ \hat{\boldsymbol{a}}_2 $.
%

\begin{proposition} \label{Thm:EstimateOfNorm}
Let $ \left\{ p_1, \ldots, p_D \right\} $ be a basis for $ \P_d\left(\mathbb{R}^2\right) $ with $\underset{i=1,\dots,D}{\max} \left\lVert p_i \right\rVert_{\infty} \leq 1$. Then,
\begin{equation}
\label{eq:BoundThm}
    \left\lVert \hat{\Pi}_d^{\mathrm{Pad}} \right\rVert_{\mathrm{op}} \leq \zeta_d + \eta_d,
\end{equation}
where
\begin{equation*}
\zeta_d \doteq \left\lVert R_{1}^{-1} \right\rVert_{1} \left(M \left\lVert Q^{T} \right\rVert_{1} + \left\lVert R_{2}\right\rVert_{1}\eta_d\right), \quad     \eta_d \doteq \left\lVert(A^T A)^{-1} A^T \right\rVert_1\left(D + M \left\lVert  W_1 R_{1}^{-1}Q^T\right\rVert_1  \right).
\end{equation*}
\end{proposition}

\begin{proof}
For any continuous function $f\in C(\Omega)$, we have
	\begin{equation} \label{eq:comoda}
\left\lVert \hat{\Pi}_d^{\mathrm{Pad}} [f] \right\rVert_\infty = \left\lVert \sum_{i=1}^D \hat{{a}}_i (f) p_i \right\rVert_\infty \leq \sum_{i=1}^D \left\lvert \hat{{a}}_i (f) \right\rvert = \left\Vert \hat{\boldsymbol{a}} (f) \right\Vert_1 = \left\Vert \hat{\boldsymbol{a}}_1 (f) \right\Vert_1 + \left\Vert \hat{\boldsymbol{a}}_2 (f) \right\Vert_1.
		\end{equation}
Using the relations~\eqref{a1eq} and~\eqref{a2eq}, and applying the triangular inequality, we have
\begin{equation*}
    \left\lVert \hat{\boldsymbol{a}}_1 (f) \right\rVert_1\le \left\lVert R_{1}^{-1}\right\rVert_{1} \left(\left\lVert Q^T \right\rVert_1 \left\lVert \boldsymbol{d} (f) \right\rVert_1 + \left\lVert R_{2}\right\rVert_1 \left\lVert \hat{\boldsymbol{a}}_2 (f)\right\rVert_1\right) 
\end{equation*}
and
\begin{equation*}\label{a2}
   \left\lVert \hat{\boldsymbol{a}}_2 (f)\right\rVert_1\le\left\lVert\left(A^{T}A\right)^{-1}A^{T}\right\rVert_1 \left\lVert \boldsymbol{b}_1 (f) \right\rVert_1.
\end{equation*}
It remains to bound $\left\lVert \boldsymbol{b} (f) \right\rVert_1$ and $\left\lVert \boldsymbol{d} (f) \right\rVert_1$. By the mean value theorem, 
one has
\begin{equation*}\label{aus1}
    \left\lVert \boldsymbol{b} (f) \right\rVert_1=\sum_{i=1}^D \left\lvert \mu_i(f)\right\rvert\le\sum_{i=1} ^D \frac{1}{|t_i|}\int_{t_i}\left\lvert f(x,y)\right\rvert \de x \de y \le \left\Vert f \right\Vert_\infty \sum_{i=1} ^D 1 = D \left\Vert f \right\Vert_\infty
\end{equation*}
and 
\begin{equation*}
\left\lVert \boldsymbol{d} (f) \right\rVert_1=\sum_{i=1}^{M} \left\lvert \mu_i(f)\right\rvert\le \sum_{i=1}^M \frac{1}{|t_i|} \int_{t_i}\left\lvert f(x,y)\right\rvert \de x \de y \le \left\Vert f \right\Vert_\infty \sum_{i=1} ^{M} 1 = M \left\Vert f \right\Vert_\infty.
\end{equation*}
Since
\begin{equation*}
    \left\lVert \hat{\Pi}_d^{\mathrm{Pad}} \right\rVert_{\mathrm{op}} = \sup_{\left\Vert f \right\Vert_\infty \leq 1} \left\lVert \hat{\Pi}_d^{\mathrm{Pad}} [f] \right\rVert_\infty, 
\end{equation*}
plugging the above expressions into the bounds for $ \left\lVert \hat{\boldsymbol{a}}_1 (f) \right\rVert_1 $ and $ \left\lVert \hat{\boldsymbol{a}}_2 (f) \right\rVert_1 $, and taking the supremum over all $f \in C(\Omega)$ with $ \left\Vert f \right\Vert_\infty \leq 1 $, we get $ \left\lVert \hat{\boldsymbol{a}}_1 \right\rVert_1 \leq \zeta_d $ and $ \left\lVert \hat{\boldsymbol{a}}_2\right\rVert_1 \leq \eta_d $. Finally, substituting this into \eqref{eq:comoda}, we get the desired bound in \eqref{eq:BoundThm}.
\end{proof}

\begin{figure}[ht]
	\centering
	\includegraphics[width = 0.6\textwidth]{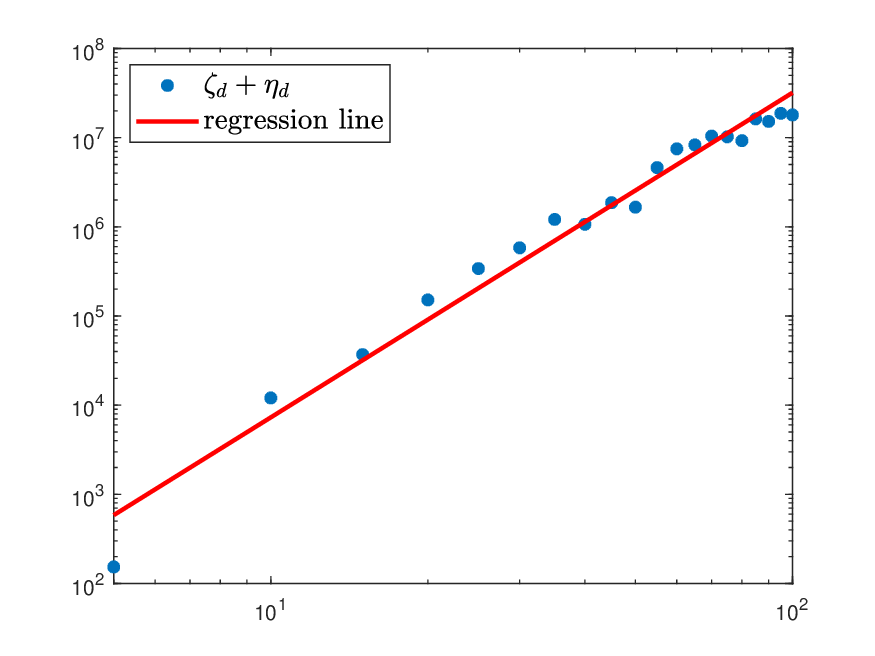}
	\caption{Trend of the bound $ \zeta_d + \eta_d $ in \eqref{eq:BoundThm} for increasing numbers $N$ of triangles in the Friedrichs-Keller triangulation $\mathcal{T}_{N}^{\mathrm{reg}}$. This trend is computed using the largest integer $m$ satisfying~\eqref{conditionOnm} and the degree $d = m + \sqrt{m}$.}
	\label{G1G2}
\end{figure}

%
%

\begin{remark}
Note that the quantities $\zeta_d$ and $\eta_d$ and, thus, the quality of the estimate \eqref{eq:BoundThm} depend strongly on the selected basis. Both, the Chebyshev basis and the monomial basis, satisfy the hypothesis of Proposition \ref{Thm:EstimateOfNorm}. However, while for the monomial basis the matrices $ C $ and $ W $ are highly ill-conditioned, it is known that the Chebyshev basis leads to rather well-conditioned Vandermonde-type matrices. For this, the Chebyshev basis seems to be a good candidate to calculate the bounds $\zeta_d$ and $\eta_d$.
In Figure~\ref{G1G2}, we report the right-hand bounds \eqref{eq:BoundThm} (blue points) for the norm of the histopolation-regression operator for different Friedrichs–Keller triangulations $\mathcal{T}_{N}^{\mathrm{reg}}$, $N = 2n^2$, $n\in\mathbb{N}$, in the square $\Omega=[-1,1]^2 $. The values are computed using the largest integer $m$ satisfying~\eqref{conditionOnm} and the degree $d = m + \sqrt{m}$. 
From this graphic, we observe that 
     \begin{equation*}
     \zeta_d + \eta_d \approx e^{3.64 \log n}=n^{3.64}. 
     \end{equation*}
\end{remark}
%
%

\subsection{Bounding the Lebesgue constant in the histopolation setting}
In the pure histopolation setting, we can derive more concrete bounds for the Lebesgue constant, as in this case also the norms $\|\Pi_m^{\mathrm{Pad}}\|_{\mathrm{op}}$ can be characterized more explicitly. These characterizations are well-known in the classical interpolation setting where the Lebesgue constant can be stated in terms of the Lagrange basis functions. Such characterizations have been found also in more abstract frameworks \cite{ARRLeb}. For a collection $ \mathcal{T}_M = \{ t_1, \ldots, t_M \} $ of triangles that form a unisolvent set for $ \P_m \left(\mathbb{R}^2\right)$, we obtain this characterization of the Lebesgue constant in terms of the quantity
\begin{equation} \label{eq:LebConst}
	\mathscr{L}_m \left(\mathcal{T}_M\right) \doteq \sup_{t \subset \Omega} \frac{1}{|t|} \sum_{i=1}^M \left| \int_t \ell_{t_i}(x,y) \de x \de y \right|  = \sup_{\xi \in \Omega} \sum_{i=1}^M \left| \ell_{t_i}(\xi)\right|, 
\end{equation} 
where $  t $ ranges over all triangles supported in $ \Omega $ and $ \{ \ell_{t_1}, \ldots, \ell_{t_M} \} $ is the Lagrange basis for the space $\mathbb{P}_m(\mathbb{R}^2)$ satisfying the histopolation conditions $ \frac{1}{|t_i|} \int_{t_i} \ell_{t_j}(x,y) \de x \de y = \delta_{i,j} $. While the first definition in \eqref{eq:LebConst} corresponds to the definition originally given in \cite{ARRLeb} for differential forms, the second is a consequence of the mean value theorem and resembles the definition of the Lebesgue constant in the nodal setting. The following identity for the Lebesgue constant can be derived from \cite{BruniElefante}.

\begin{proposition} \label{prop:intnorm}
Let $\mathcal{T}_M = \{ t_1, \ldots, t_M \} $ be a unisolvent set of triangles for $ \P_m \left(\mathbb{R}^2\right)$ and suppose that the triangles in $\mathcal{T}_M$ are disjoint in the sense that the measure $\left| t_i \cap t_j\right| = 0 $ for each $ t_i,t_j \in \mathcal{T}_M $, $t_i\ne t_j$. Then, the operator norm of the respective histopolation operator $\Pi_m$ satisfies
$$ \left\Vert \Pi_m \right\Vert_{\mathrm{op}} = \mathscr{L}_m \left(\mathcal{T}_M\right). $$
    \end{proposition}

In order to calculate the quantity $\mathscr{L}_m \left(\mathcal{T}_M\right)$, and hence also the operator norm $\left\Vert \Pi_m \right\Vert_{\mathrm{op}}$, we require a scheme to compute the Lagrange basis of the space $ \P_m \left(\mathbb{R}^2\right)$ for a unisolvent set $\mathcal{T}_M$ of triangles.
Despite general explicit expressions for the Lagrange basis are not available (see \cite{bruni2024polynomial} for an account), we can calculate the Lagrange basis numerically in terms of a chosen basis $ \{ p_1, \ldots, p_M \} $ of $ \P_m\left(\mathbb{R}^2\right) $. Once this basis is selected, it is sufficient to invert the \emph{generalized Vandermonde matrix}%
\begin{equation} \label{eq:VandermondeMatrix}
	[V_{i,j}] \doteq [\mu_i(p_j)] = \left[ \frac{1}{|t_i|} \int_{t_i} p_j(x,y) \de x \de y \right] \in \mathbb{R}^{M \times M},
\end{equation}
to recover the expansion coefficients for the Lagrange basis functions. More precisely, we have
$$ \int_{t_i} \sum_{k=1}^M V_{k,j}^{-1} p_k (x,y) \de x \de y = \delta_{i,j} ,$$
meaning that $ \ell_j (x,y) = \sum_{k=1}^M V_{k,j}^{-1} p_k (x,y) $ provides the desired Lagrange basis functions. The Vandermonde matrix \eqref{eq:VandermondeMatrix} is invertible if and only if the functionals $\mu_i(f)$ are linearly independent, see \cite[Lemma 3.2.2]{AtkinsonHan}. In case of the Padua triangles $\T_M^{\mathrm{Pad}}$, the invertibility of the matrix \eqref{eq:VandermondeMatrix} is guaranteed under the assumptions of Proposition \ref{prop:stability}.

To calculate the Lebesgue constant $\mathscr{L}_m \left(\mathcal{T}_M\right)$ we can now proceed as follows. We consider the rightmost term in the definition \eqref{eq:LebConst} of the constant $ \mathscr{L}_m \left(\mathcal{T}_M\right) $ 
and define the vector
$ \boldsymbol{p} (\xi) \doteq \left[ p_1 \left(\xi\right), \ldots, p_M (\xi) \right] $.
Then, using the inverse of the Vandermonde matrix $V$ to describe the Lagrange basis functions, we can rewrite $\mathscr{L}_m \left(\mathcal{T}_M\right)$ more compactly as
	\begin{equation} \label{eq:shortLeb}
		\mathscr{L}_m \left(\mathcal{T}_M\right) = \sup_{\xi \in \Omega} \left\Vert \boldsymbol{p} (\xi) \, V^{-1}  \right\Vert_1 .
	\end{equation}
	
To estimate the Lebesgue constant $\mathscr{L}_m \left(\mathcal{T}_M\right)$ numerically, we discretize $ \Omega $ by a suitable large discrete set of nodes $ \mathcal{X} = \{ \xi_1, \ldots, \xi_N\} $ and then replace $ \Omega $ by $ \mathcal{X} $ in \eqref{eq:shortLeb}. The whole computation can be implemented efficiently by evaluating the infinity norm of a product of two matrices. 

In Figure \ref{fig:LebPD}, the numerically calculated Lebesgue constant associated with the Padua triangles $\mathcal{T}_M^{\mathrm{Pad}}$ is shown. It is visible that the Lebesgue constant $\mathscr{L}_m \left(\mathcal{T}_M^{\mathrm{Pad}} \right)$ increases only slowly for a growing total degree $m$. The following proposition states that this growth is at most logarithmic raised to the power of two if the size of the triangles in $\mathcal{T}_M^{\mathrm{Pad}}$ is sufficiently small.  

\begin{proposition} \label{prop:stability2}  
	If the maximal triangle length $h_{\max}$ is small enough according to the bound in Proposition \ref{prop:stability}, the Lebesgue constant of the Padua triangles can be bounded by
	$$ \mathscr{L}_m \left( \mathcal{T}_M^{\mathrm{Pad}} \right) \leq \frac{K}{1-\alpha} \left(\ln m\right)^2,$$
	where $0 < \alpha < 1$ and $K>0$ is linked to the Lebesgue constant of the Padua nodes.  
\end{proposition}

\begin{proof}
    As in the proof of Proposition \ref{prop:stability}, the mean value theorem implies that the histopolation operator $\Pi_{m}^{\mathrm{Pad}} [f] = \Pi_m(\mathcal{T}_M^{\mathrm{Pad}}) [f]$ providing the polynomial histopolant of the function $f \in C(\Omega)$ on the Padua triangles $\mathcal{T}_M^{\mathrm{Pad}}$ can be read as a nodal interpolation operator $\Pi_m(\Xi_M(f))$$ [f]$ of the function $f$ on points $ \Xi_M(f) = \{ \xi_1(f), \ldots, \xi_M(f) \} $ that depend on the function $f$ and are contained in $\mathcal{T}_M^{\mathrm{Pad}}$, i.e., $ \xi_i(f) \in t_i $ for some Padua triangle $t_i \in \mathcal{T}_M^{\mathrm{Pad}}$. Since the distance between a point $\xi_i(f)$ and the corresponding Padua point in the triangle $t_i$ is smaller than the maximal length $ \max_i \ \mathrm{diam} \left(t_i\right) \leq h_{\max} $, the assumption on $h_{\max}$ ensures that \cite[Corollary 1]{PiazzonVianello} can be applied and that the norm of the interpolation operators associated with $ X_M^{\mathrm{Pad}} $ and $\Xi_M$ satisfy
	\begin{equation} \label{eq:helpme}
	\frac{\left\Vert \Pi_m(\Xi_M(f)) \right\Vert_{\mathrm{op}}}{\left\Vert \Pi_m(X_M^{\mathrm{Pad}}) \right\Vert_{\mathrm{op}}} \leq \frac{1}{1-\alpha}.  
	\end{equation} 
	The equality $ \left\Vert \Pi_m(X_M^{\mathrm{Pad}}) \right\Vert_{\mathrm{op}} = \mathscr{L}_m \left(X_M^{\mathrm{Pad}}\right) $ is classical and follows from the characterization of the nodal Lebesgue constant as the norm of the corresponding nodal interpolator. As the Lebesgue constant $ \mathscr{L}_m \left(X_M^{\mathrm{Pad}}\right) $ for the Padua points can be bounded as $\mathscr{L}_m \left(X_M^{\mathrm{Pad}}\right) \leq K \left(\ln m\right)^2 $, and \eqref{eq:helpme} holds true for any function $f \in C(\Omega)$, we get
    \[ \left\Vert \Pi_m(\mathcal{T}_M^{\mathrm{Pad}}) \right\Vert_{\mathrm{op}} \leq \sup_{ \| f \|_{\infty} \leq 1} \left\Vert \Pi_m(\Xi_M(f)) \right\Vert_{\mathrm{op}} \leq \frac{K}{1-\alpha} \left(\ln m\right)^2. \]
    The identity $ \left\Vert \Pi_m(\mathcal{T}_M^{\mathrm{Pad}}) \right\Vert_{\mathrm{op}} = \mathscr{L}_m \left(\mathcal{T}_M^{\mathrm{Pad}}\right) $ is guaranteed by Proposition \ref{prop:intnorm}, as Padua triangles do not intersect in their interior.
\end{proof}

The trend for the Lebesgue constant $\mathscr{L}_m \left( \mathcal{T}_M^{\mathrm{Pad}} \right) $ is depicted in Figure \ref{fig:LebPD}.
The left subfigure indicates that the Lebesgue constant is primarily affected by the degree $m$ of the polynomial space and only in a minor way on the number $N$ of triangles. Further, for increasing $N$ the Lebesgue constant of the Padua triangles tends to the Lebesgue constant of the Padua points. This gets also visible on the right hand side of Figure \ref{fig:LebPD}, where the trend of the Lebesgue constant of the Padua triangles shows a consistent behavior with that of the Padua points, cf. \cite{Bos;2006:BLI}.

\begin{figure}[ht]
	\centering
    \includegraphics[width=0.49\textwidth]{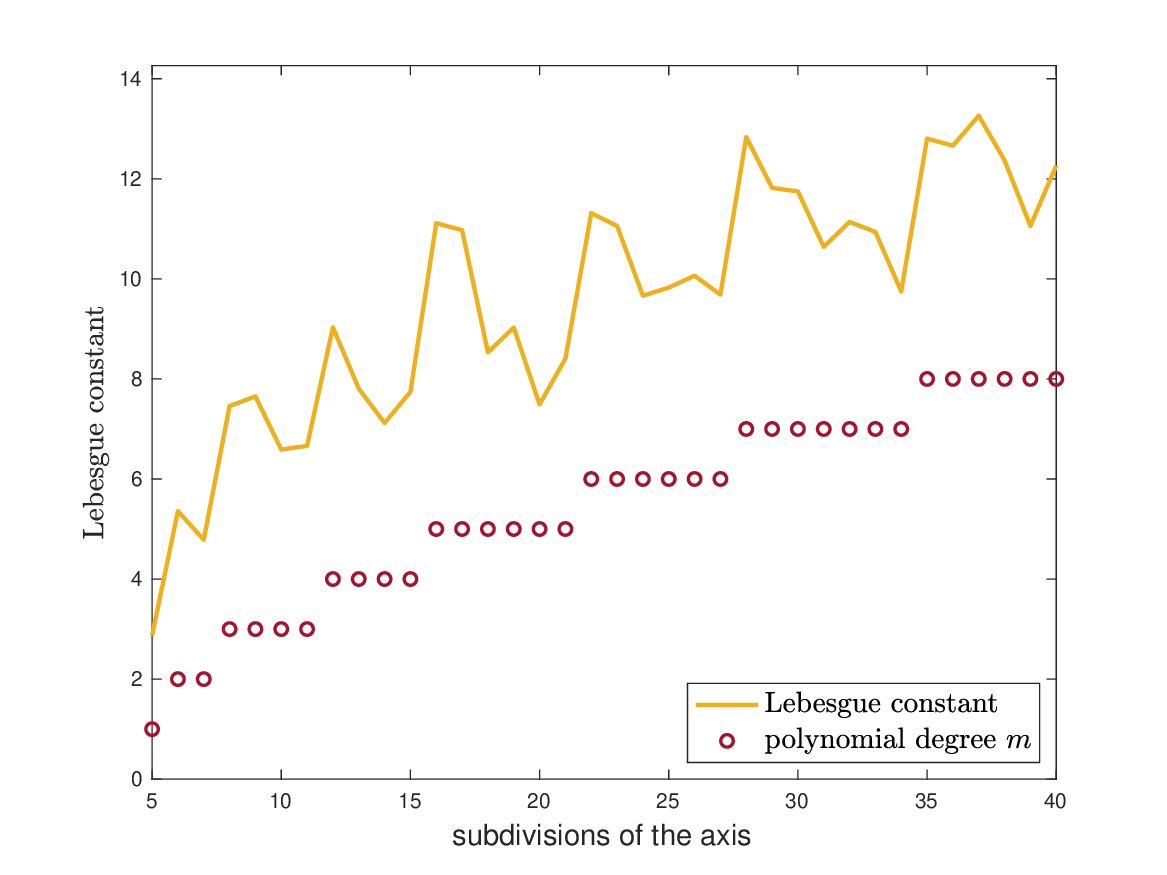}
    \includegraphics[width=0.49\textwidth]{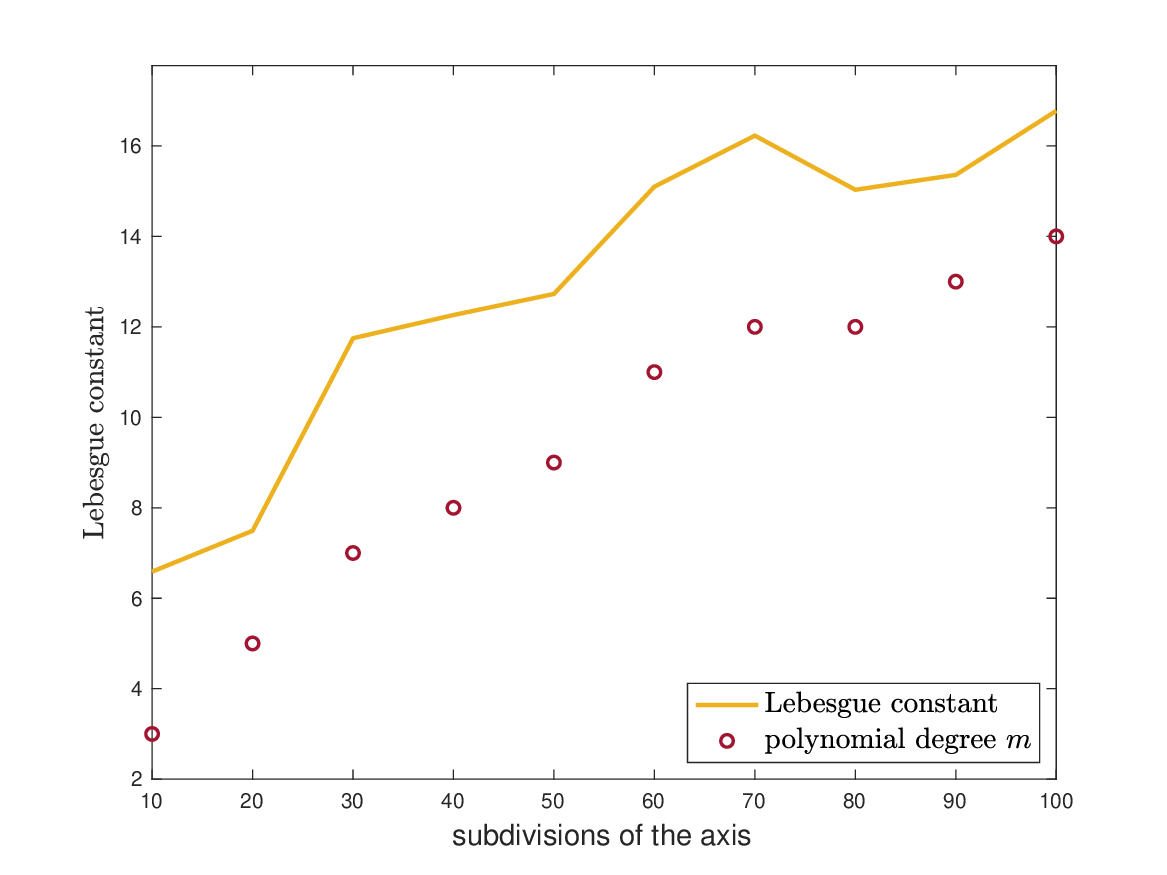}
	\caption{Trend of the Lebesgue constant defined in Eq. \eqref{eq:LebConst} for Padua triangles. The number of Friedrichs-Keller triangles is $ N = 2n^2 $, where $n$ is the number of subdivisions of each axis, and the corresponding polynomial degree is $ m $, here shown by red circles.} \label{fig:LebPD}
\end{figure}

\begin{remark}
The theoretical and computational methods developed for triangles in this article can be naturally extended, with small modifications, to other partitionings of the domain $\Omega$, such as 
in \cite{BruniElefante}. In order to obtain reasonable convergence results, attention has to be paid 
to the regularity and the uniformity of the elements in the partitions.

Moreover, the approaches presented in this work can be readily generalized also to higher dimensions using, for instance, simplices instead of triangles. Additionally, beyond the use of Padua points, other explicit point sets with slowly increasing Lebesgue constants such as the Xu points \cite{Xu1996}, the Lissajous node points \cite{Dencker2017,Dencker2017b,Erb2016Degenerate,Erb2016NonDegenerate}, or similar node sets as in \cite{Harris2013} are equally suitable. In the next section, we will further explore alternative strategies based on the Fekete and the Leja triangles.
\end{remark}

\section{Fekete triangles, Leja sequences and their implementation}
\label{sec3new}
In the nodal framework, a common strategy for constructing sets with a low Lebesgue constant consists in looking for Fekete nodes or Leja sequences of nodes. While the theoretical determination of such sets is rather tough \cite{BruniErbFekete}, some linear algebra based algorithms for their approximate computation have been studied \cite{bos2010computing}. We extend these greedy methods to our histopolation setting, and note that, in contrast to Padua triangles, Leja and approximate Fekete triangles may be extracted also from very irregular meshes.

\subsection{Fekete and approximate Fekete triangles}
Fekete nodes are points that maximize the determinant of the Vandermonde matrix. In the univariate case on the interval $[-1,1]$, they coincide with the Legendre-Gauss-Lobatto nodes, and their Lebesgue constant exhibits logarithmic growth \cite{IbrahimogluSurvey}. In higher-dimensional settings, Fekete points maintain a small growth of the Lebesgue constant as the number of interpolation points increases. However, explicit formulas for these points are available only in a few specific cases (see, for instance, \cite{Bos2023}) and one has to stick to numerical approximations, as discussed in \cite{BSV12}. For histopolation, Fekete segments within the interval $[-1,1]$ have been introduced and analyzed in \cite{BruniErbFekete}. Also in this setting, the Lebesgue constant for the Fekete segments grows at a very slow rate. 

Similar as for Fekete segments in $[-1,1]$, one can define a set of Fekete triangles $ \mathcal{T}_M^{\mathrm{Fek}} $ in $\Omega$ as the set of triangles $$\mathcal{T}_M^{\mathrm{Fek}} =\left\{ t_1^{\mathrm{Fek}}, \ldots, t_M^{\mathrm{Fek}} \right\}$$ 
such that the Vandermonde determinant
\begin{equation} \label{eq:FeketeDet}
	\det \left[ V_{i,j} \right] = \det \left[ \frac{1}{|t_i|} \int_{t_i} p_j(x,y) \de x \de y \right], \quad i,j \in \{1, \ldots, M\}, 
	\end{equation}
is maximized over all possible sets $\{t_1, \ldots, t_M\}$ of triangles in $\Omega$.

Similar as in the segmental case \cite{BruniErbFekete}, one can generally show that the corresponding Lebesgue constant grows, at most, as the cardinality of the space $ \P_m\left(\mathbb{R}^2\right)  $:
\begin{equation} \label{eq:LebFek}
	\mathscr{L}_{m} \left(\mathcal{T}_M^{\mathrm{Fek}}\right) \leq \dim\left( \P_m \left(\mathbb{R}^2\right)\right) .
\end{equation}
This inequality arises from basic linear algebra principles, and remains independent of whether a linear functional represents a point evaluation or an average over a domain. Moreover, it does not depend on the domains used for averaging the function. 

To make the maximization of the Vandermonde determinant \eqref{eq:FeketeDet} computationally feasible, we will, as in the case of the Padua triangles, extract approximate Fekete triangles from an initial triangulation $\T_N$. To further reduce the computational cost, we will, as described in \cite{BSV12}, compute these approximate Fekete triangles $\mathcal{T}_M^{\mathrm{Fek}}$ from the triangulation $\T_N$ by applying a column-pivoted QR decomposition to the Vandermonde matrix, which corresponds to an iterative greedy algorithm for the selection of the triangles. The respective algorithm for the numerical extraction of the approximate Fekete triangles, which is almost identical to the nodal one, is given in Algorithm \ref{algFek}. In Figure \ref{LejaeFekTriangle} (left) approximate Fekete triangles extracted from a regular Friedrichs-Keller triangulation are displayed. Note that as the approximate Fekete triangles $\mathcal{T}_M^{\mathrm{Fek}}$ are extracted from a finite collection $\mathcal{T}_N$ the estimate \eqref{eq:LebFek} for the Lebesgue constant may in fact not hold true anymore. This gets visible also in the numerical evaluations of the Lebesgue constants given in Figure \ref{fig:LebLF} and Figure \ref{fig:LebLFdetail}. 

\begin{algorithm}[h!]
	\caption{Extraction of approximate Fekete triangles}
	\begin{algorithmic}[1]
		\Require a triangulation $\mathcal{T}_N= \{t_1,\dots,t_N\}$ of $ \Omega $, a basis $ \left\{p_1, \dots,p_M \right\}$ for $ \P_m \left(\mathbb{R}^2\right) $
		\State Compute the rectangular collocation matrix $ W \in \mathbb{R}^{N \times M}$ given as
        \[ [W_{i,j}] = \left[ \mu_i(p_j)\right] = \left[ \frac{1}{|t_i|} \int_{t_i} p_j(x,y) \de x \de y \right].\]
		\State Calculate the column pivoted QR decomposition of $W^T$ as 
        \[Q R = W^T P, \]
        with a permutation matrix $P \in \mathbb{R}^{N \times N}$. 
        \State Store the first $M$ indices $I = \{i_1, \ldots, i_M\}$ linked to the permutation $P$ of the columns of the matrix $W^T$. 
        \Ensure the set of approximate Fekete triangles $\mathcal{T}^{\mathrm{Fek}}_M= \left\{ t_{i_1},\dots,t_{i_M} \right\}$. 
	\end{algorithmic}
	\label{algFek}
\end{algorithm}

\subsection{Discrete Leja sequences of triangles}
 Instead of performing a global optimization for approximate Fekete triangles based on the Vandermonde determinant in \eqref{eq:FeketeDet}, one can adopt an alternative approach by ordering the basis $ \left\{ p_1, \ldots, p_M \right\} $ of the polynomial space $ \P_m\left(\mathbb{R}^2\right) $ and constructing a sequence $ \mathcal{T}_M^{\mathrm{Leja}} $ of triangles $ \left\{t_1^{\mathrm{Leja}}, \ldots, t_M^{\mathrm{Leja}}\right\} $. At each iteration step $ k \in \{1, \ldots, M\} $, this methods selects a new triangle that maximizes the determinant of a progressively increasing Vandermonde matrix, following the idea of Leja sequences, as for instance described in \cite{DeMarchiLeja}. 
 
 The greedy algorithm for extracting a discrete Leja sequence from the triangulation $\T_N$, similar as in the nodal case, is presented in Algorithm \ref{algLeja}. Figure \ref{LejaeFekTriangle} (right) shows a set of Leja triangles extracted from a regular Friedrichs-Keller triangulation.

\begin{algorithm}[h!]
	\caption{Extraction of discrete Leja triangles}
		\begin{algorithmic}[1]
		\Require a triangulation $\mathcal{T}_N= \{t_1,\dots,t_N\}$ of $ \Omega $, a basis $ \left\{p_1, \dots,p_M \right\}$ for $ \P_m \left(\mathbb{R}^2\right) $
		\State Compute the rectangular collocation matrix $ W \in \mathbb{R}^{N \times M}$ given as
        \[ [W_{i,j}] = \left[ \mu_i(p_j)\right] = \left[ \frac{1}{|t_i|} \int_{t_i} p_j(x,y) \de x \de y \right].\]
		\State Calculate the LU decomposition of $W$ with partial pivoting such that 
        \[L R  = P W, \]
        with a permutation matrix $P \in \mathbb{R}^{N \times N}$. 
        \State Extract the first $M$ indices $I = \{i_1, \ldots, i_M\}$ linked to the permutation $P$ of the rows of the matrix $W$. 
        \Ensure the set of discrete Leja triangles $\mathcal{T}^{\mathrm{Leja}}_M= \left\{ t_{i_1},\dots,t_{i_M} \right\}$. 
	\end{algorithmic}
	\label{algLeja}
\end{algorithm}

\begin{figure}[ht]
	\centering
	\includegraphics[width=0.49\textwidth]{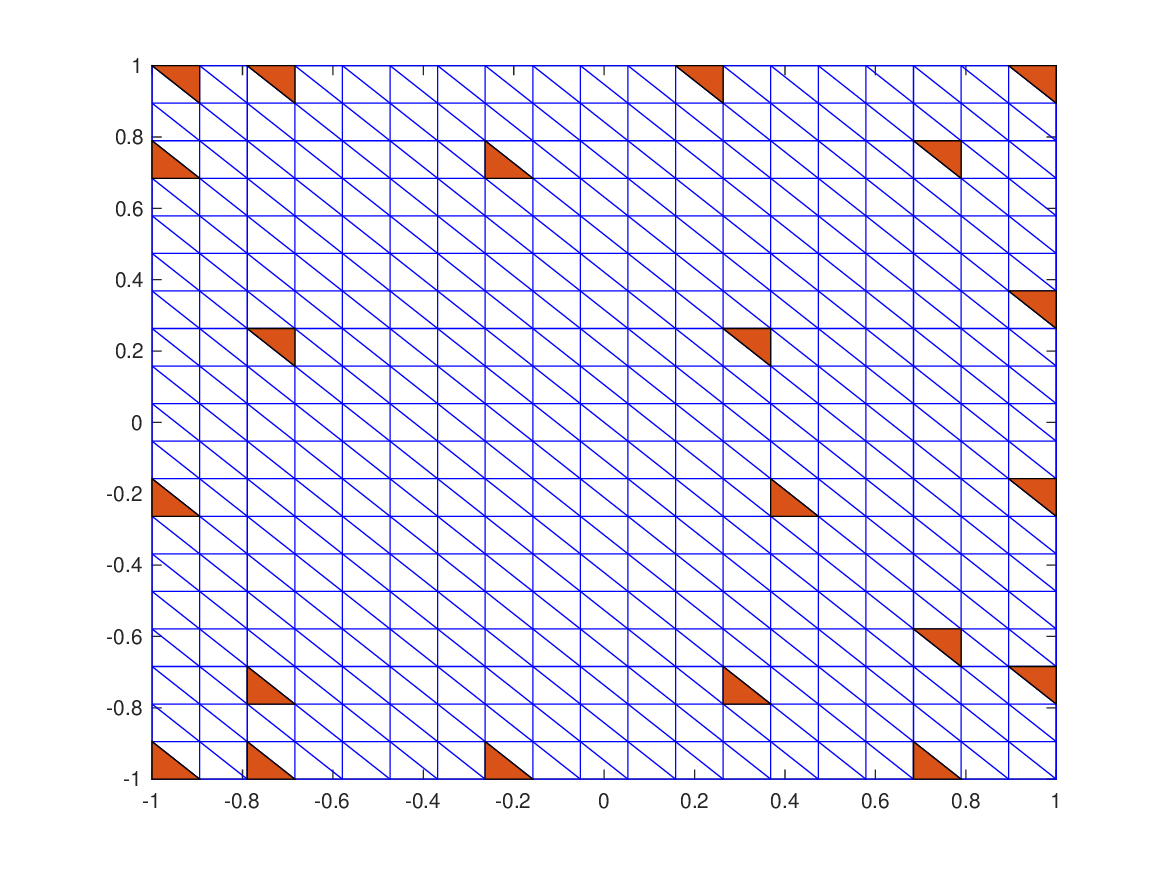}
	\includegraphics[width=0.49\textwidth]{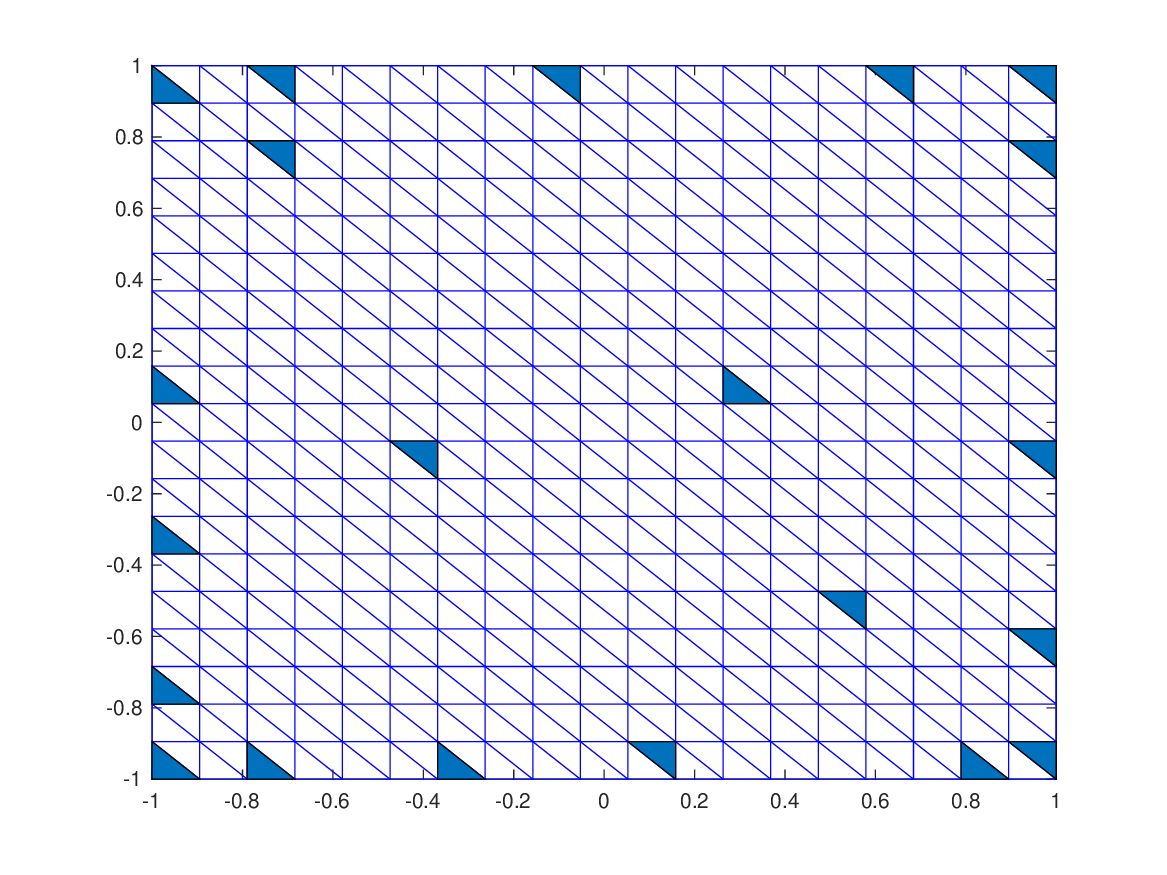}
	\caption{Approximate Fekete (left) and discrete Leja triangles (right) extracted from the regular Friedrichs-Keller triangulation $\mathcal{T}^{\mathrm{reg}}_{800}$ using $n = 20$ subintervals on each axis.}
	\label{LejaeFekTriangle}
\end{figure}

%


In Figure \ref{fig:LebLF}, the behavior of the Lebesgue constant \eqref{eq:LebConst} for approximate Fekete triangles and discrete Leja sequences of triangles is visualized. The respective computations are based on Vandermonde matrices using tensor-product Chebyshev polynomials as a basis. The selection of Fekete and Leja triangles requires a predetermined choice of the total polynomial degree $ m $. To be consistent with the choice of the Padua triangles, we set $ m $ to be the largest admissible value for the Padua triangles as established in Proposition \ref{thmimp}. 

However, for the Fekete and Leja triangles, this choice of $m$ may sometimes be too large relative to the overall size of the triangulation $\T_N$, leading to a clear violation of the estimate in Eq. \eqref{eq:LebFek}. This phenomenon is evident in Figure \ref{fig:LebLF} for certain isolated values $m$, and becomes even more pronounced at a finer discretization level, as shown in Figure \ref{fig:LebLFdetail}. Despite these specific outliers, the approximate Fekete triangles yield a slightly better Lebesgue constant than the discrete Leja triangles. However, when compared to the Padua triangles (cf. Figure \ref{fig:LebPD}), both Fekete and Leja triangles seem to exhibit an inferior performance. 

\begin{figure}[ht]
	\centering
	\includegraphics[width=0.49\textwidth]{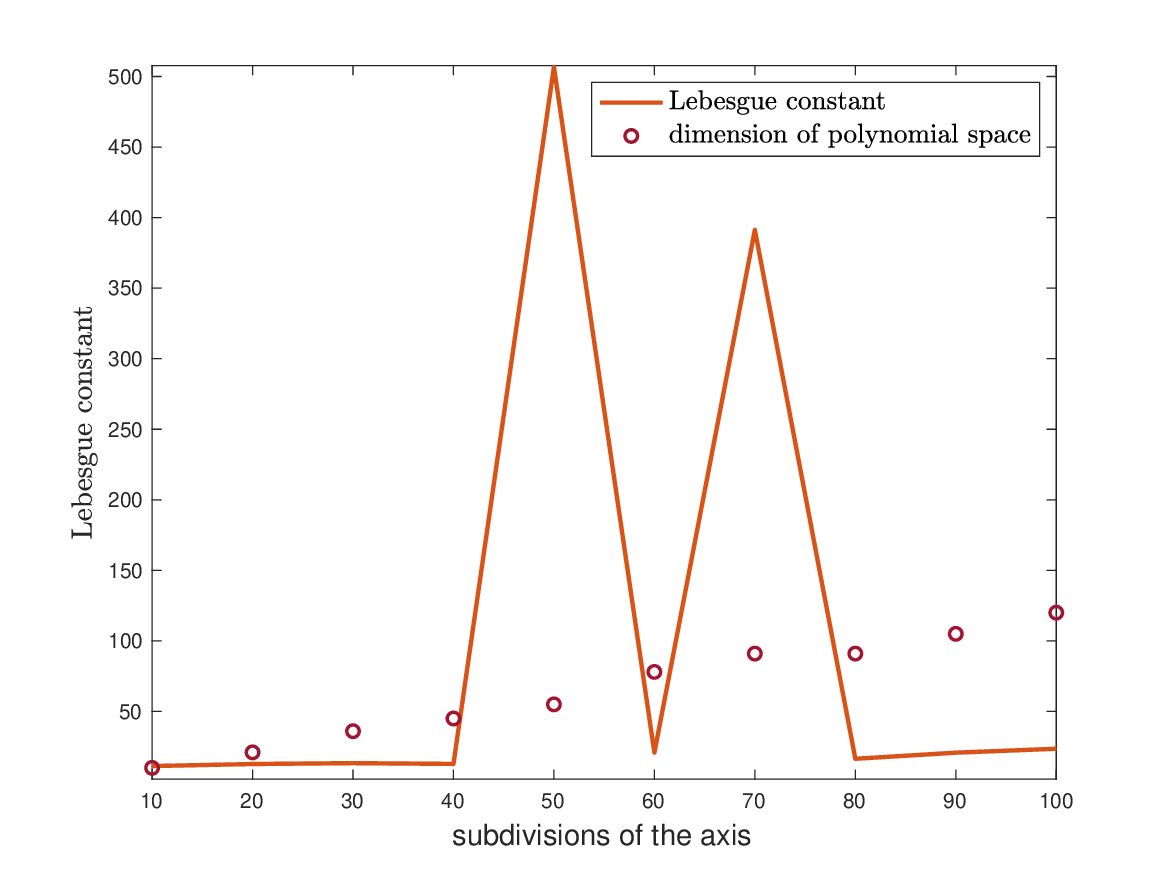}
	\includegraphics[width=0.49\textwidth]{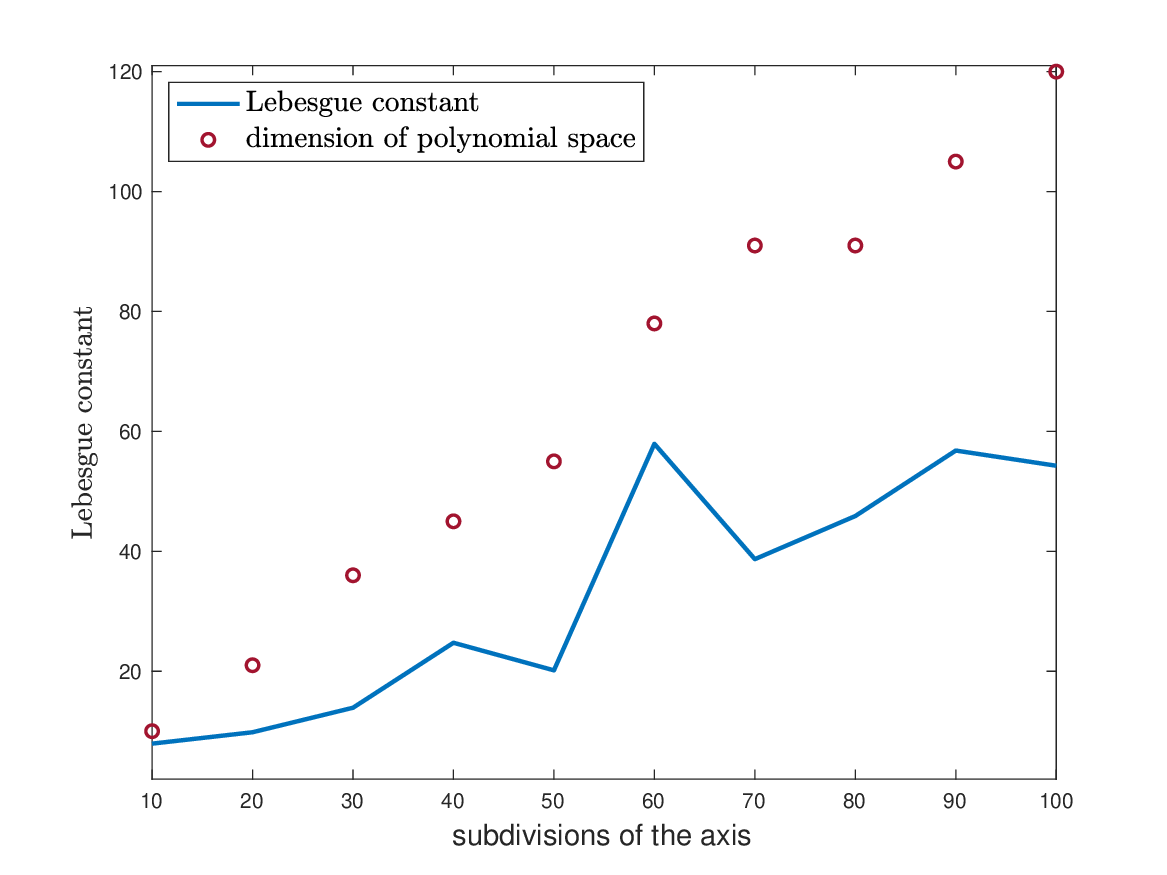}
	\caption{Lebesgue constants of approximate Fekete (left) and discrete Leja triangles (right) computed on the regular triangulation $\mathcal{T}^{\mathrm{reg}}_{N}$, where $ N = 2 n^2$ and $ n \in \{ 10, 20, \ldots, 100 \}$ is the number of subdivisions of each axis.}
	\label{fig:LebLF}
\end{figure}

\begin{figure}[ht]
	\centering
	\includegraphics[width=0.49\textwidth]{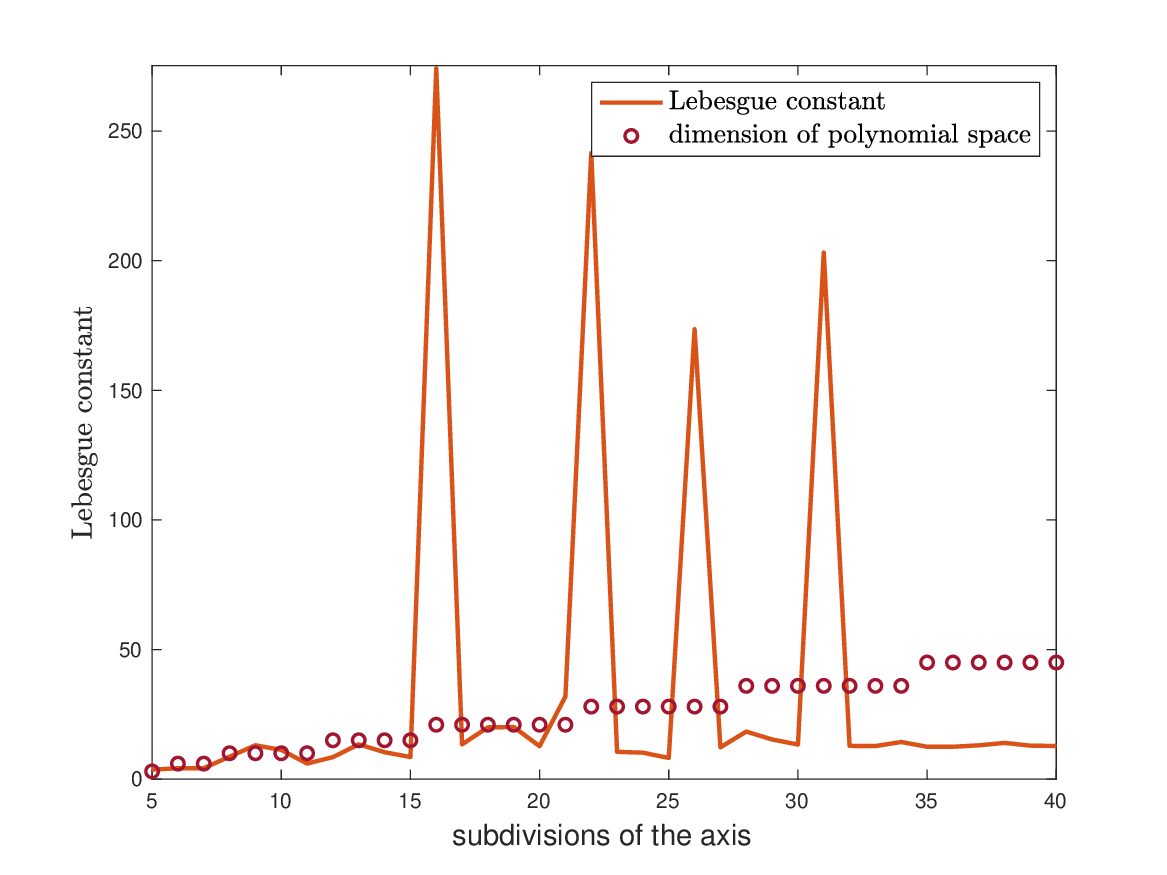}
	\includegraphics[width=0.49\textwidth]{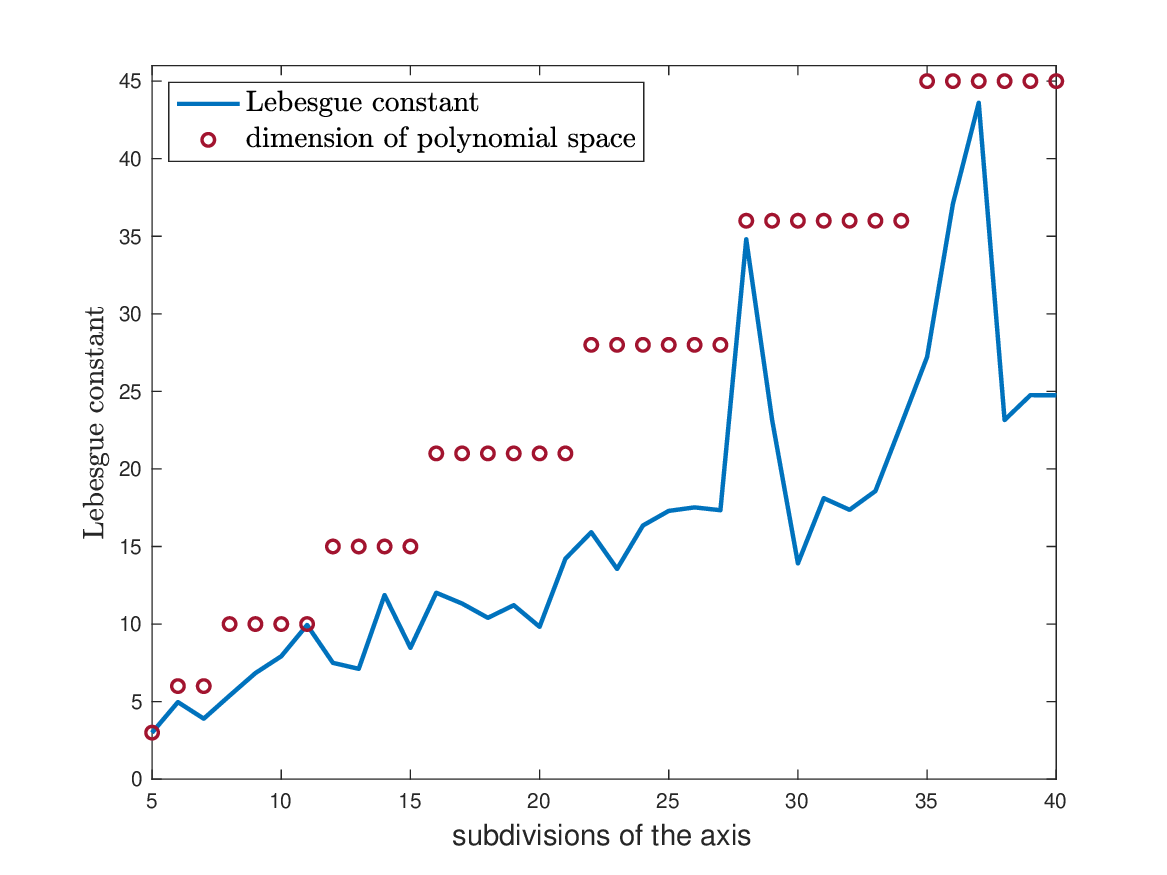}
	\caption{Lebesgue constants of Fekete (left) and Leja triangles (right) computed on the regular triangulation $\mathcal{T}^{\mathrm{reg}}_{N}$. Here, all subdivision numbers $n \in\{5, 6, \ldots, 40 \}$ are considered. }
	\label{fig:LebLFdetail}
\end{figure}

\subsection{Calculation of histopolation and histopolation-regression operators}\label{sec3}

In analogy to the histopolation operator \eqref{ea:histopolationPadua} and the histopolation-regression operator \eqref{operatorPhat} for the Padua triangles, we can define corresponding projection operators using the approximate Fekete triangles $\T^{\mathrm{Fek}}_{M}$ and the discrete Leja triangles $\T^{\mathrm{Leja}}_{M}$ instead. While the Padua triangles require as reference domain the square $\Omega = [-1,1]^2$, the Fekete and Leja triangles can be extracted from triangulations $\T_N$ of a general polygonal domain. 

In case of the approximate Fekete triangles $\mathcal{T}_M^{\mathrm{Fek}}$, fixing a basis $\{p_1, \ldots p_M\}$ of $\mathbb{P}_m\left(\mathbb{R}^2\right)$, the histopolation operator ${\Pi}^{\mathrm{Fek}}_{m}$ is given as 
\begin{equation*}\label{operatorinterpFek}
\begin{array}{rcl}
{\Pi}^{\mathrm{Fek}}_{m}: f \in C\left(\Omega\right) &\rightarrow& {\Pi}^{\mathrm{Fek}}_{m}[f] =\sum\limits_{j=1}^{M}a_j p_j\in \mathbb{P}_m\left(\mathbb{R}^2\right),
\end{array}
\end{equation*}
where the vector of coefficients $\left[a_1,\dots,a_{M}\right]^T$ are determined such that the histopolation conditions $\mu_i(f) = \mu_i\left(\sum_{j=1}^{M}a_j p_j \right)$, $i \in \{1, \ldots, M\}$, are satisfied on the Fekete triangles $\mathcal{T}_M^{\mathrm{Fek}}$. On the other hand, for the histopolation-regression problem, we get the projection operator
\begin{equation}\label{operatorPhatFek}
\begin{array}{rcl}
\hat{\Pi}^{\mathrm{Fek}}_{d}: f \in C\left(\Omega\right) &\rightarrow& \hat{\Pi}^{\mathrm{Fek}}_{d}[f] = \sum\limits_{j=1}^{D}\hat{a}_j p_j \in \mathbb{P}_d\left(\mathbb{R}^2\right). 
\end{array}
\end{equation}
where $d > m$ and $\hat{\boldsymbol{a}}=\left[\hat{a}_1,\dots,\hat{a}_{D}\right]^T$ is the solution of the augmented linear system~\eqref{LagrangeMultmethod} relative to the solution of the constrained least-squares problem involving exact histopolation on $\mathcal{T}_M^{\mathrm{Fek}}$ and regression on the remaining triangles. In this case, to set up the system of equations, the triangulation $\T_N$ gets reordered such that $\mathcal{T}_M^{\mathrm{Fek}}$ denote the first $M$ triangles in $\T_N$. 
 In a completely analogous way, the polynomial histopolation operator and the histopolation-regression operator on the Leja triangles $\mathcal{T}_M^{\mathrm{Leja}}$ can be introduced using the notation
\begin{equation*}\label{operatorinterpLeja}
\begin{array}{rcl}
{\Pi}^{\mathrm{Leja}}_{m}: f \in C\left(\Omega\right) &\rightarrow& {\Pi}^{\mathrm{Leja}}_{m}[f] = \sum\limits_{j=1}^{M}a_j p_j\in \mathbb{P}_m\left(\mathbb{R}^2\right),
\end{array}
\end{equation*}
and
\begin{equation}\label{operatorPhatLeja}
\begin{array}{rcl}
\hat{\Pi}^{\mathrm{Leja}}_{d}: f\in C\left(\Omega\right) &\rightarrow& \hat{\Pi}^{\mathrm{Leja}}_{d}[f]=\sum\limits_{j=1}^{D} \hat{a}_j p_j \in \mathbb{P}_d\left(\mathbb{R}^2\right). 
\end{array}
\end{equation}

We notice that the properties shown for the Padua histopolation-regression operator~\eqref{operatorPhat}, including the bound for the Lebesgue constant, can be adapted also for the operators~\eqref{operatorPhatFek} and~\eqref{operatorPhatLeja}. On the other hand, the quasi-optimal logarithmic bound of the Lebesgue constant for histopolation on the Padua triangles seen in Proposition \ref{prop:stability2} seems not to be true for the approximate Fekete and discrete Leja triangles (cf. Figure \ref{fig:LebLF} and Figure \ref{fig:LebLFdetail}).  

We also note that while for the Padua triangles Proposition \ref{prop:stability} gives a theoretic guarantee for the unisolvence of the histopolation problem, the unisolvence in case of the Fekete and Leja triangles has to be checked numerically. This can be done implicitly within the pivoted QR or LU factorization in the extraction of the respective triangles.

Algorithm \ref{algfinal} summarizes the calculation of the histopolation-regression polynomial $\hat{\Pi}_d^{\mathrm{Pad}}[f]$ for the Padua triangles. The respective adaption to the operators $\hat{\Pi}^{\mathrm{Fek}}_{d}$ and $\hat{\Pi}^{\mathrm{Leja}}_{d}$ defined in~\eqref{operatorPhatFek} and~\eqref{operatorPhatLeja} is immediate. 

\begin{algorithm}[h!]
	\caption{histopolation-regression operators}
	\begin{algorithmic}[1]
		\Require \begin{itemize}
		    \item[a)] $m\in\mathbb{N}$,  $d\in\mathbb{N}$, a triangulation $\mathcal{T}_N= \{t_1,\dots,t_N\}$ of $ \Omega $,
            \item[b)] a basis $ \left\{p_1, \dots,p_D \right\}$ for $ \P_d \left(\mathbb{R}^2\right) $ with $\left\{p_1,\dots,p_M\right\}$ being a basis for $\mathbb{P}_m\left(\mathbb{R}^2\right)$,
            \item[c)] the data vector $\boldsymbol{b} = [\mu_1(f), \ldots, \mu_N(f)]^T$.
		\end{itemize}

\vspace{2mm}
        \State Compute the Padua triangles $\T_M^{\mathrm{Pad}}$ by Algorithm~\ref{algPad}.
        \State Reorder $\mathcal{T}_N$ such that $\T_M^{\mathrm{Pad}}$ are the first $M$ elements of $\mathcal{T}_N$, reorder respectively $\boldsymbol{b}$ such that the first $M$ entries $\boldsymbol{d}$ contain the data values on the triangles $\T_M^{\mathrm{Pad}}$.  
        \State Compute the collocation matrices $ W $ and $ C $ (using a Gaussian quadrature rule)
		and create the augmented matrix $\begin{bmatrix}
		2 W^T W & C^T  \\
		C & 0  \\
	\end{bmatrix}$ for the linear system~\eqref{LagrangeMultmethod}.
		\State Compute the coefficient vector $\hat{\boldsymbol{a}}=\left[\hat{a}_1,\dots,\hat{a}_D\right]^T$ by solving the linear system~\eqref{LagrangeMultmethod}.
		\State Calculate the polynomial $\sum\limits_{j=1}^D \hat{a}_j p_j$.
        \Ensure the approximation polynomial  $\hat{\Pi}_d^{\mathrm{Pad}}[f] = \sum\limits_{j=1}^D \hat{a}_j p_j$.
	\end{algorithmic}
	\label{algfinal}
\end{algorithm}

\section{Numerical results}
\label{sec4}

In the following, we provide a numerical evaluation of the performance of the histopolation and histopolation-regression methods based on the Padua, Fekete and Leja triangles. For this evaluation, we consider the following three test functions:
\begin{equation*}
    f_1(x,y)=e^{x+y} \sin{(\pi x y)},
    \qquad f_2(x,y)=|x+y|,
    \qquad f_3(x,y)=\frac{1}{1+10(x^2+y^2)}.
\end{equation*}
The first function is regular and smooth, while the second is non-differentiable at the diagonal points $x = -y$. In the second case, we expect lower convergence rates for the histopolation polynomials. The third function is the Runge function already encountered in Figure \ref{fig:Runge}. Note that $ \Vert f_1 \Vert_\infty \approx 4.71 $, $ \Vert f_2 \Vert_\infty = 2 $, and $ \Vert f_3 \Vert_\infty = 1 $. All the errors in the numerical experiments are computed with respect to the uniform norm on $[-1,1]^2$.


\subsection{Comparison of histopolators}
We compare the results obtained for the Padua triangles with the histopolators based on the approximate Fekete triangles and the discrete Leja triangles. As a rule for determining the degree $m$ of the polynomial space $\P_m\left(\mathbb{R}^2 \right)$, the number $m$ is chosen largest possible such that the well-posedness condition of Proposition \ref{thmimp} is still satisfied. Integrals are computed with a Gaussian quadrature rule. For the test function $ f_1 $, convergence of the schemes can be observed for all three histopolators, as shown in Figure \ref{fig:ConvHist}.

\begin{figure}[H]
    \centering
    \includegraphics[width=0.60\linewidth]{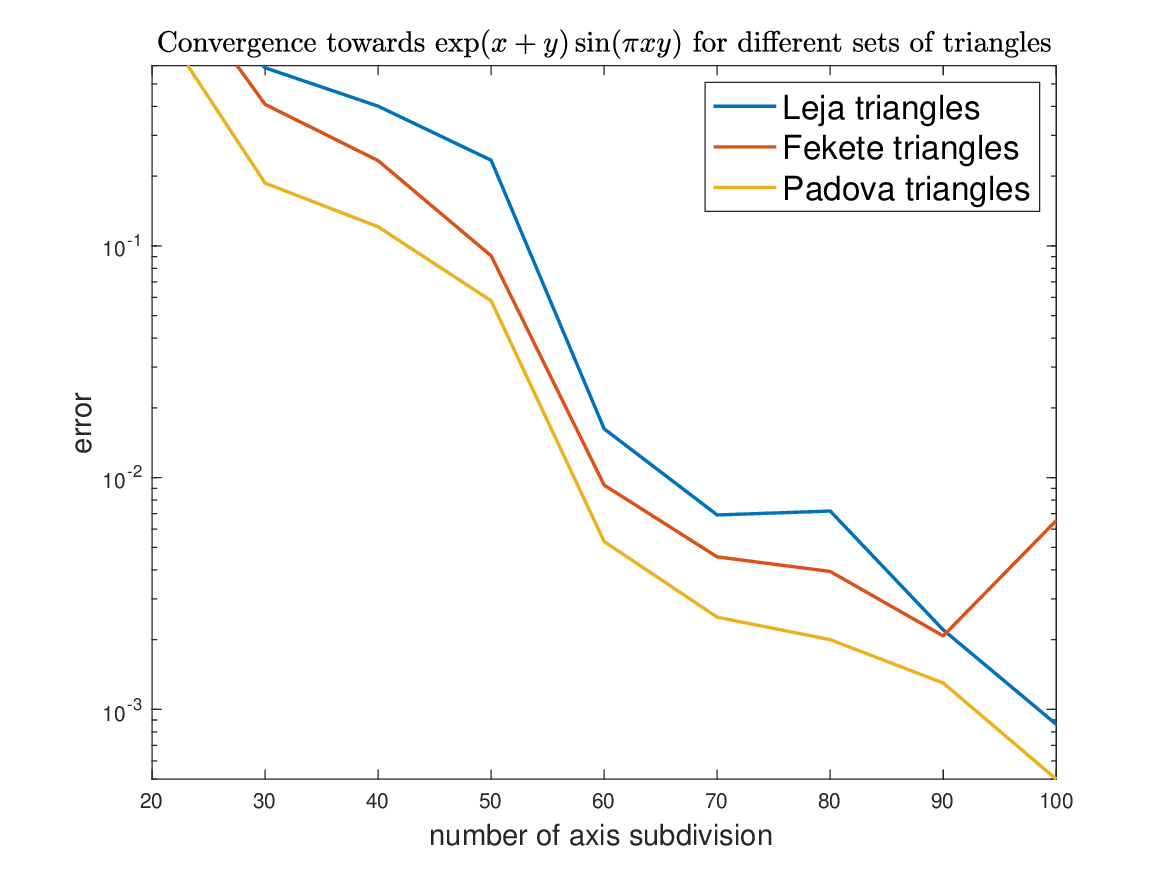} 
    \caption{Convergence of the histopolators towards the function $ f_1$.}
    \label{fig:ConvHist}
\end{figure}

In contrast, we see that the instability of the Lebesgue constant associated with the Fekete triangles (see Figure \ref{fig:LebLF}) has an impact also for the convergence of the methods, which is not visible anymore for the functions $ f_2 $ and $ f_3 $. This is depicted in Figure \ref{fig:Divergence}, where the error associated with the Fekete histopolator is dashed. A softer instability is shown by the histopolator associated with the Leja triangles. The peaks that can be seen for the Fekete (and partially for the Leja) histopolator correspond to the peaks already observed for the Lebesgue constants in Figure \ref{fig:LebLF} and Figure \ref{fig:LebLFdetail}. 

\begin{figure}[H]
    \centering
    \includegraphics[width=0.49\linewidth]{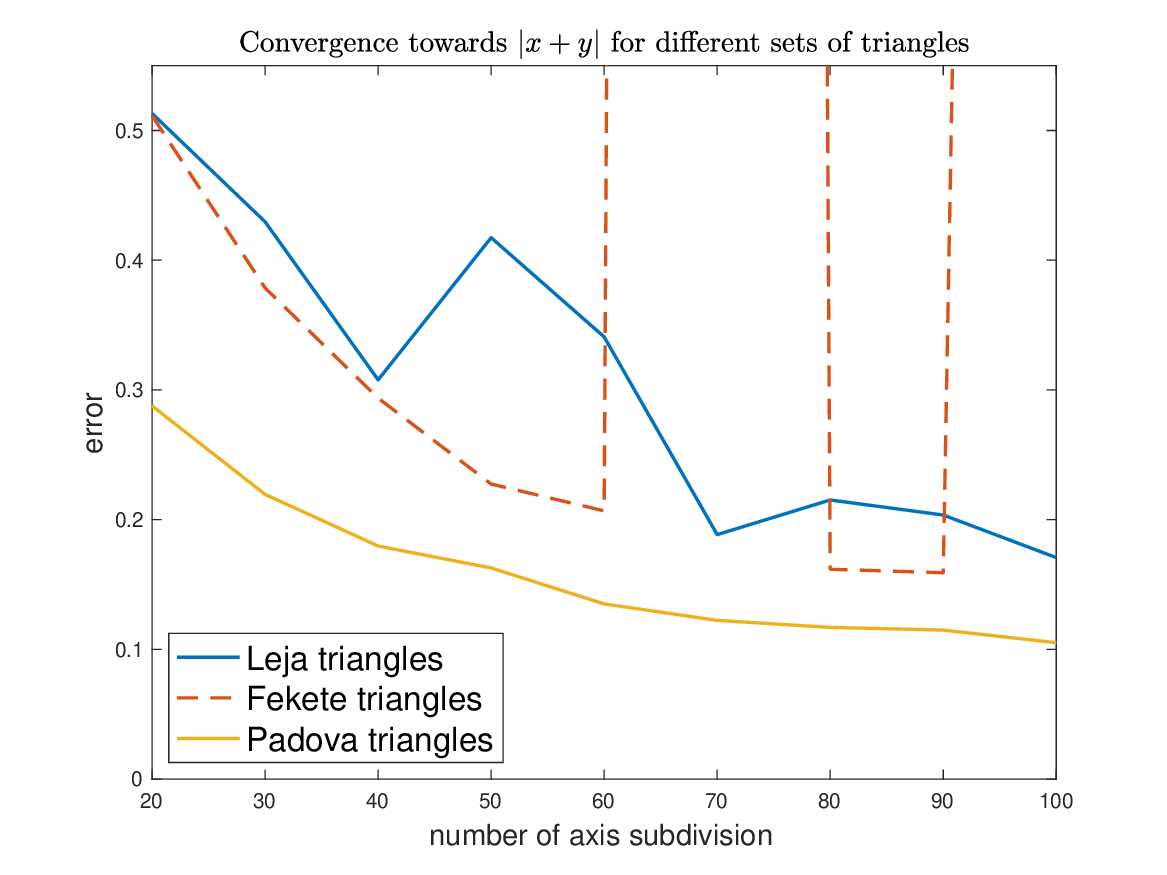}
    \
    \includegraphics[width=0.49\linewidth]{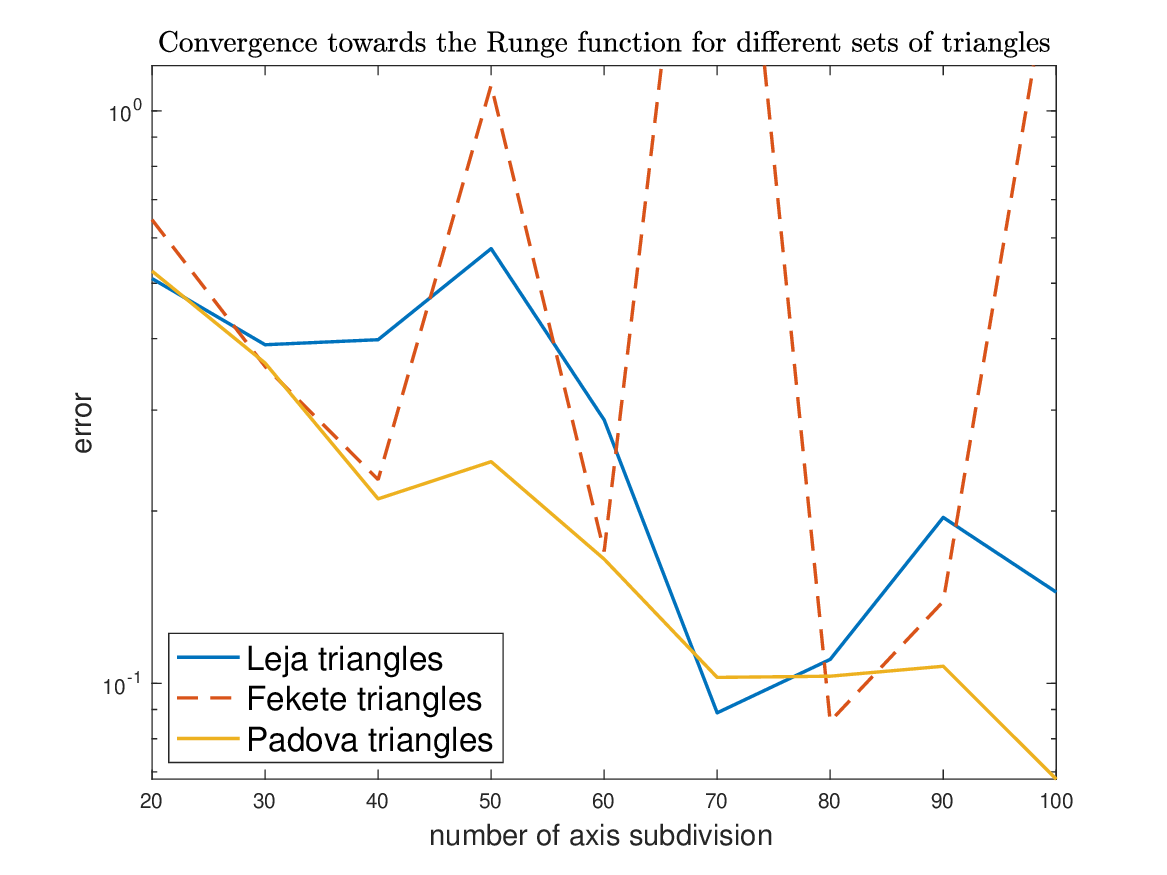}
    \caption{Convergence of the histopolators towards the functions $ f_2 $ and $ f_3 $.}
    \label{fig:Divergence}
\end{figure}



As evident from Figure \ref{fig:ConvHist}, the histopolator associated with the Padua triangles is the only one, among those discussed, that shows convergence for all three test functions. 

\subsection{Adding the regression term}

As in the previous experiment, let $m$ be the maximal natural number still satisfying the condition ~\eqref{gencond} of Proposition \ref{thmimp}. Now, we additionally assume for the combined histopolation-regression scheme, similarly as in~\cite{bruni2024polynomial}, that the polynomial degree $d$ satisfies 
\begin{equation*}
    d \sim m+\sqrt{m}.
\end{equation*}
This particular relationship between $m$ and $d$ has its origins in the nodal framework \cite{DeMarchi:2015:OTC}, and we use it for the additional regression term discussed in Section \ref{sect:regressionmethod}. The respective convergence results are presented in Figure \ref{fig:ConvRegr} for $ f_1 $ and $f_2$, and in Figure \ref{fig:ConvRegrRunge} for $ f_3 $.

\begin{figure}[H]
    \centering
    \includegraphics[width=0.49\linewidth]{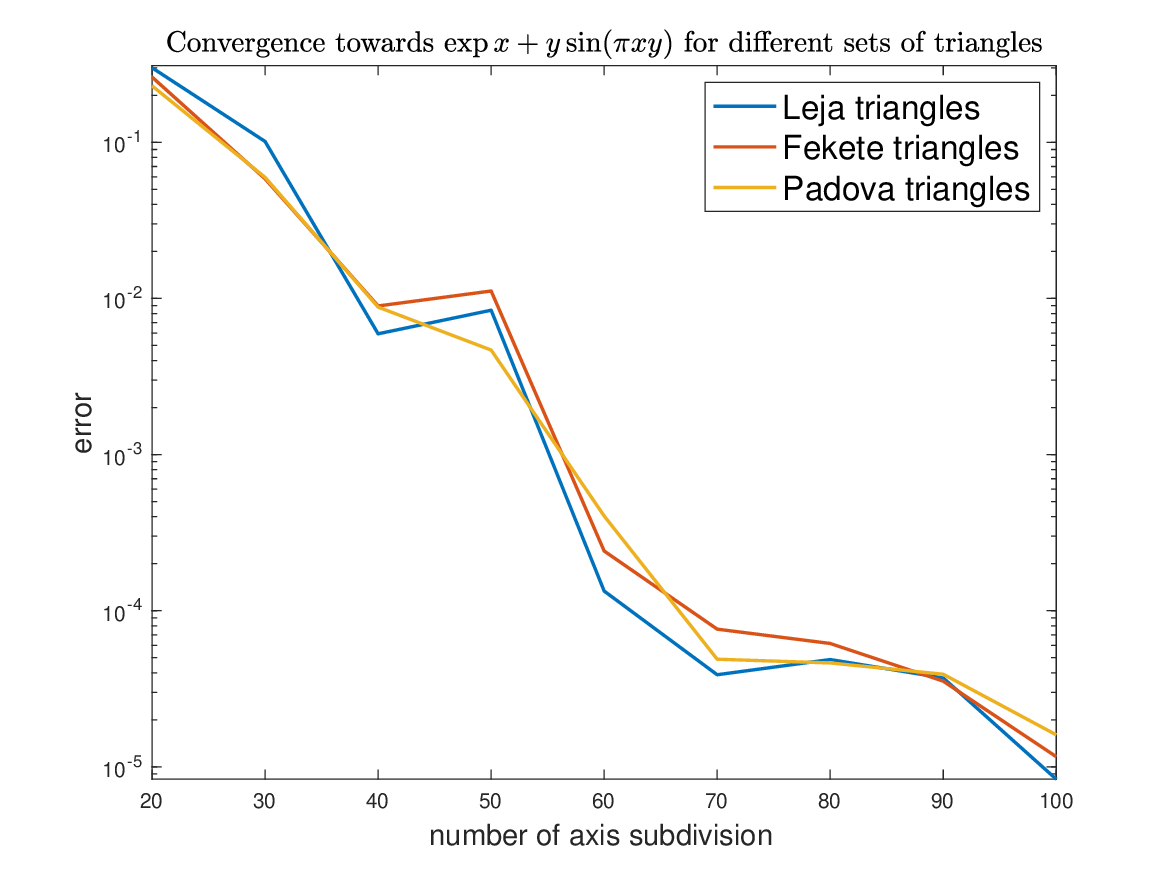} \
    \includegraphics[width=0.49\linewidth]{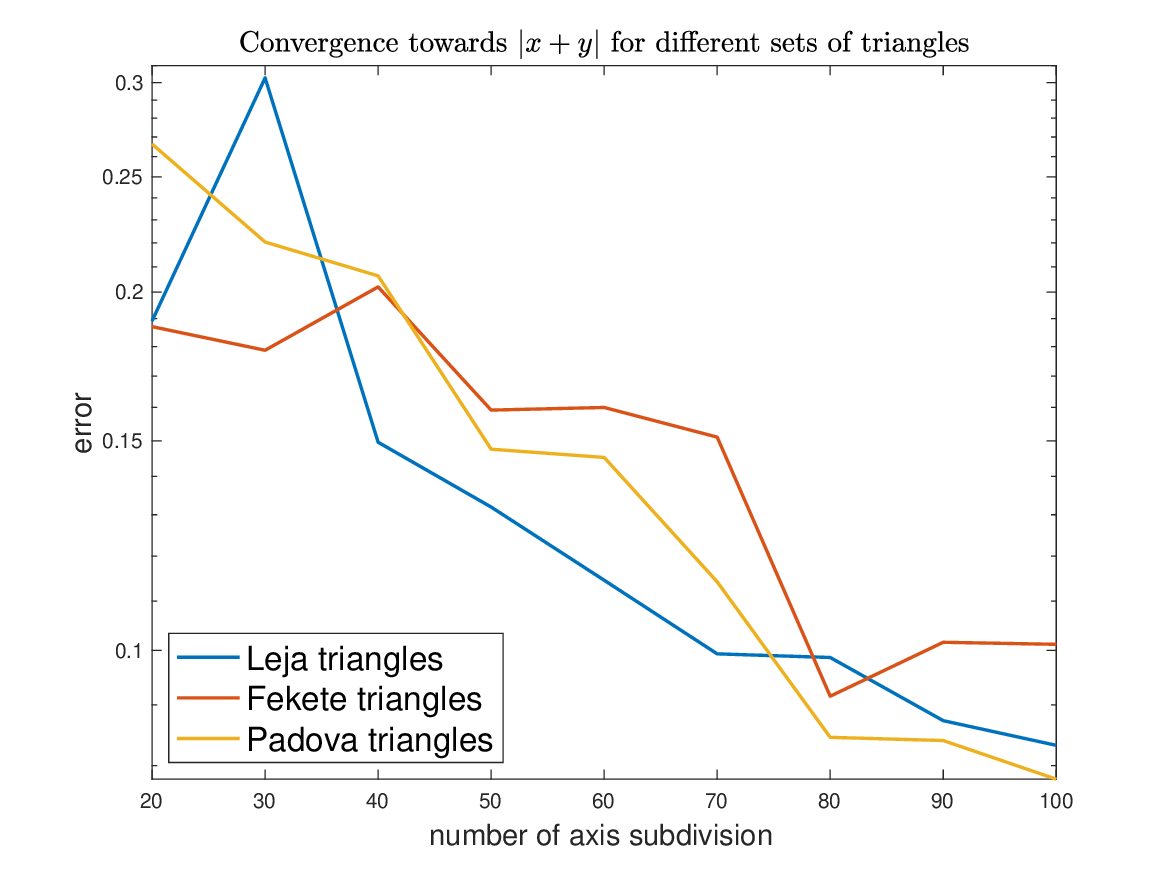}
    \caption{Convergence of the histopolation-regression approximants towards the functions $ f_1$ and $ f_2$.}
    \label{fig:ConvRegr}
\end{figure}

We notice that for the smooth function $ f_1 $ the convergence of the histopolants is improved (cf. with Figure \ref{fig:ConvHist}). Furthermore, for the function $ f_2 $, the instabilities and peaks visible in the error of the Fekete histopolation seem to be resolved. A similar effect is visible also for the Runge function $f_3$, see Figure \ref{fig:ConvRegrRunge}. 


\begin{figure}[H]
    \centering
    \includegraphics[width=0.60\linewidth]{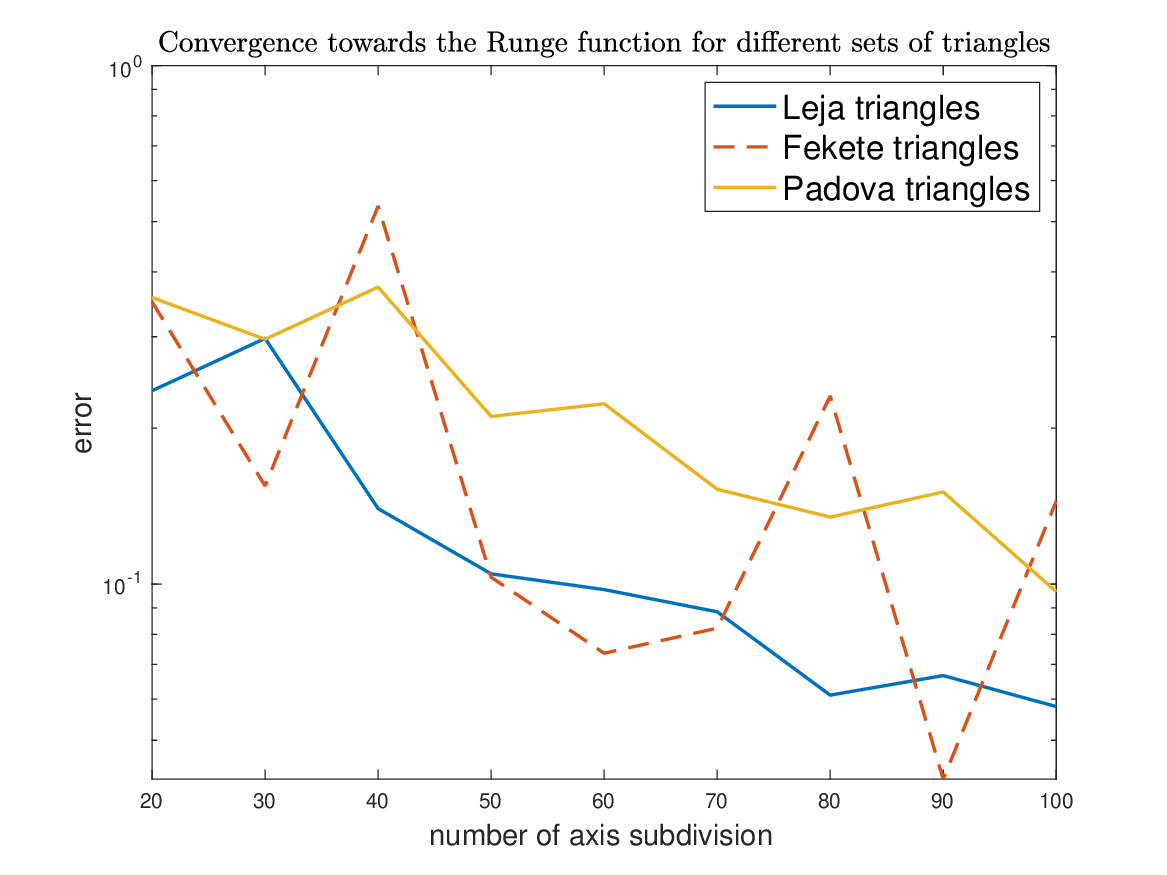}
    \caption{Convergence of the histopolation-regression approximants towards the Runge function $ f_3 $.}
    \label{fig:ConvRegrRunge}
\end{figure}

From these numerical tests, we can infer that the regression term seems to accelerate the convergence of the already convergent histopolants, as for instance seen for the smooth function $ f_1 (x,y) $. Also for the functions $ f_2 (x,y) $ and $ f_3 (x,y) $, the additional regression term seems to stabilize and improve the convergence of the approximants associated with the discrete Leja triangles and the approximate Fekete triangles.

\subsection{Random mesh}

As seen in the previous sections, the Padua triangles require a strong hypothesis on the mesh regularity in order to guarantee a well-defined histopolation scheme. In contrast, the triangles given by the greedy approaches in Section \ref{sec3new} can be computed easily and provide well-defined histopolation schemes also for less regular initial triangulations. To see this, we increase in this section the randomness of the initial mesh,  
and then apply Algorithms 1--3 to study the convergence behavior of the resulting histopolation operators.

\subsubsection{Grid from random points on the boundary}

As a first attempt, we consider a non-uniform partitioning of the axes to generate a triangular mesh, as visualized in Figure \ref{fig:MeshRandom}. Observe that very different triangles are selected by the three procedures.

\begin{figure}[H]
    \centering
    \includegraphics[width=0.60\linewidth]{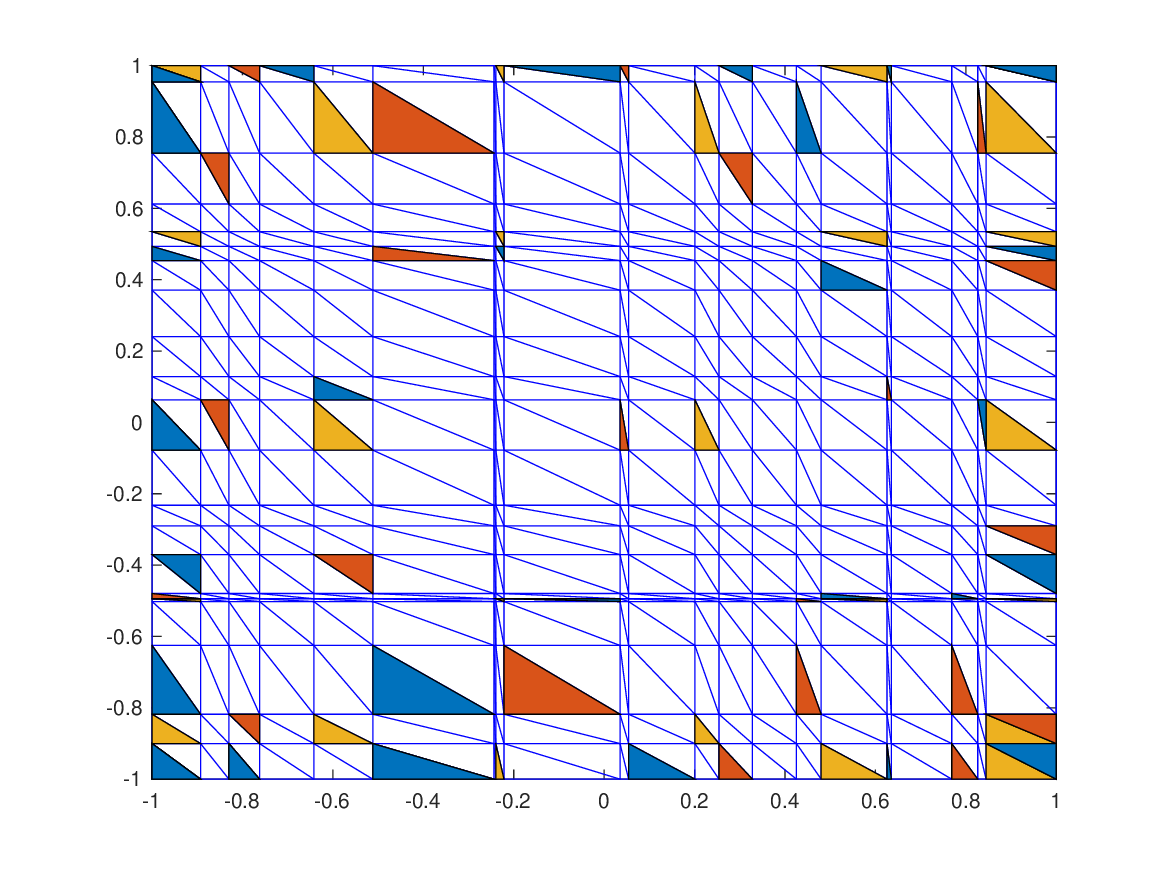}
    \caption{A randomized Friedrichs-Keller triangulation. The extracted triangles depending on the method are pictured in different colours: blue for Leja, orange for Fekete, yellow for Padua.}
    \label{fig:MeshRandom}
\end{figure}

In this experiment, approximately 30\% of the attempts failed to produce a Padua histopolator, meaning that the attribution map \eqref{eq:PPtoPT} is not well defined. Figure \ref{fig:MeshRandom} shows that the condition on $h_\mathrm{max}$ is frequently violated by this approach. Nevertheless, when produced, the histopolator associated with the Padua triangles performs efficiently. This is shown in Figure \ref{fig:ConvSemiRand}, where a comparison of the histopolation-regression approximants for the three triangle extraction procedures is presented for the functions $ f_1 $ and $ f_3 $.

\begin{figure}[H]
    \centering
    \includegraphics[width=0.49\linewidth]{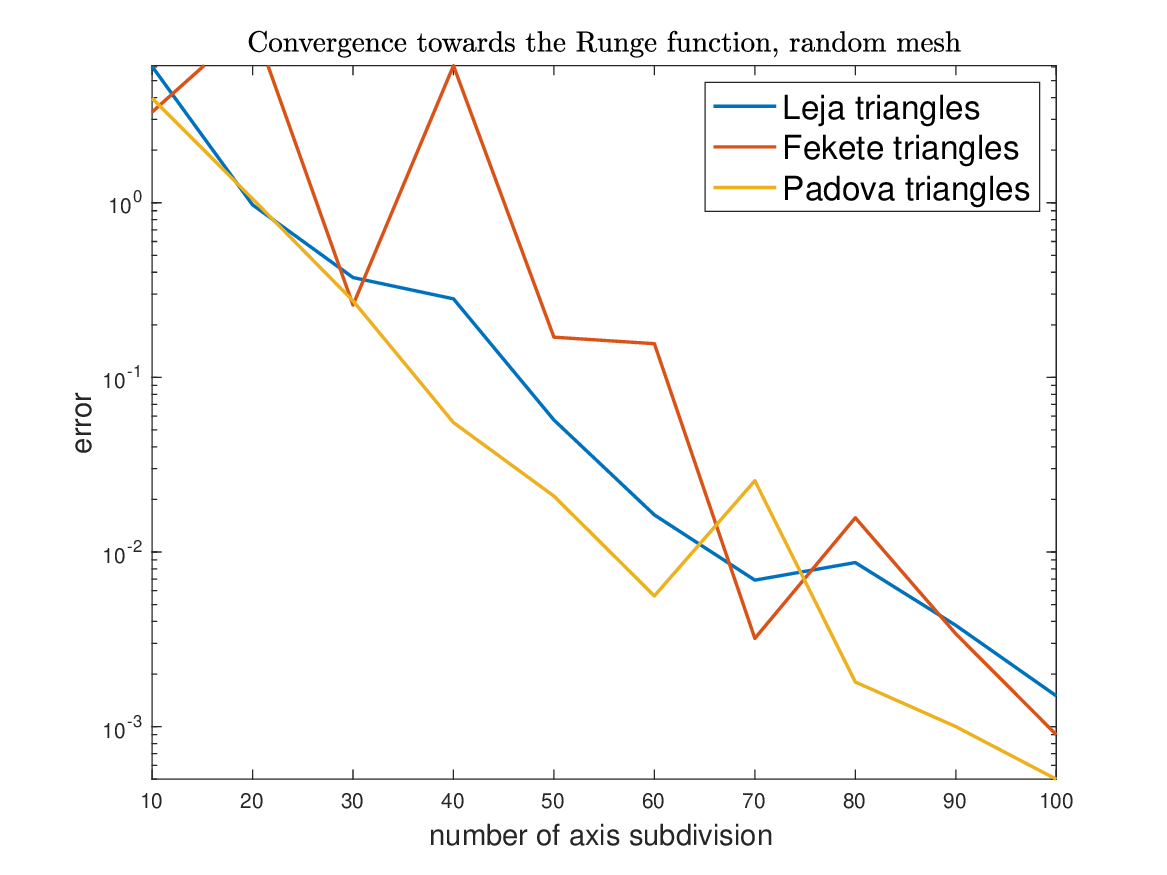} \
    \includegraphics[width=0.49\linewidth]{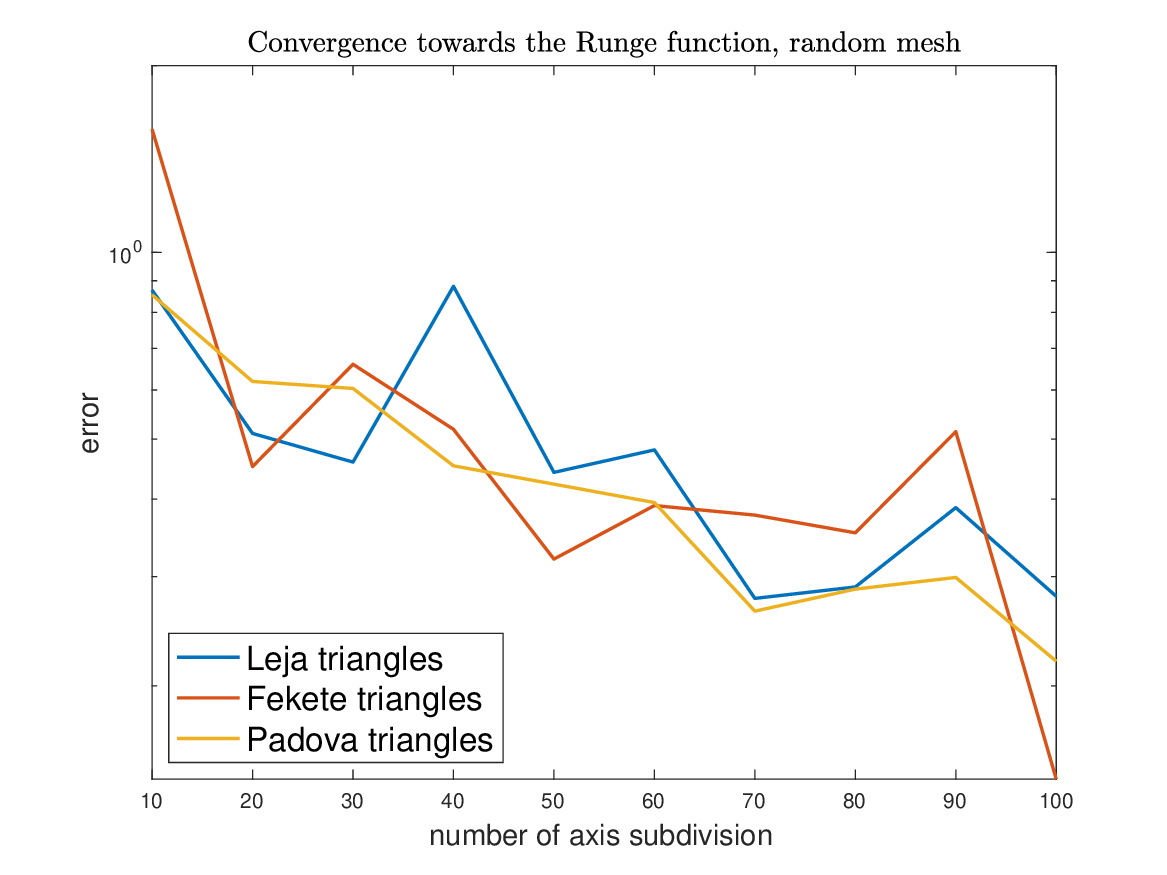}
    \caption{Convergence of the histopolation-regression approximants towards the functions $ f_1 $ and $ f_3$, for randomized meshes as given in Figure \ref{fig:MeshRandom}.}
    \label{fig:ConvSemiRand}
\end{figure}

\subsubsection{Random Delaunay triangulation}
A more involved situation emerges when the mesh is completely random, and does not stem from a random subdivision of the axes.
In this case, the attribution map \eqref{eq:PPtoPT} rarely turns out to be well defined, while the extraction of Fekete and Leja triangles is granted. An example of this is given in Figure \ref{fig:MeshRandomComparison}. In this example, the entire triangulation consists of about $2000$ elements, and according to the criteria in Section \ref{sec2}, a regular Friedrichs-Keller triangulation of this size would allow a well-defined histopolation of total degree $ m = 7 $. The dimension of the corresponding space is $ \dim \P_7 (\Omega) = 36 $, but only $ 25 $ Padua triangles (shaded in yellow) are detected. As a consequence, the histopolator is not defined. The right hand panel of that same image shows that $ 36 $ Leja triangles are correctly extracted.

\begin{figure}[H]
    \centering
    \includegraphics[width=0.49\linewidth]{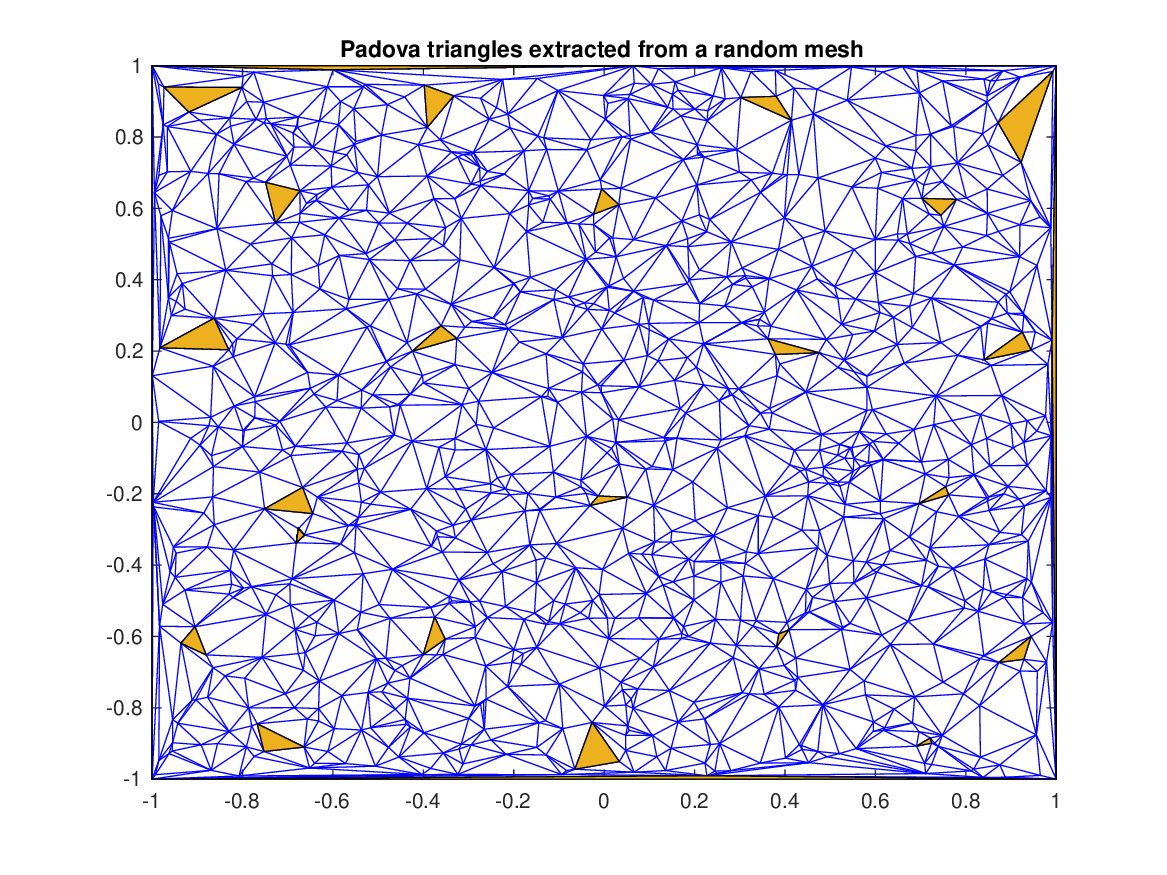} \
    \includegraphics[width=0.49\linewidth]{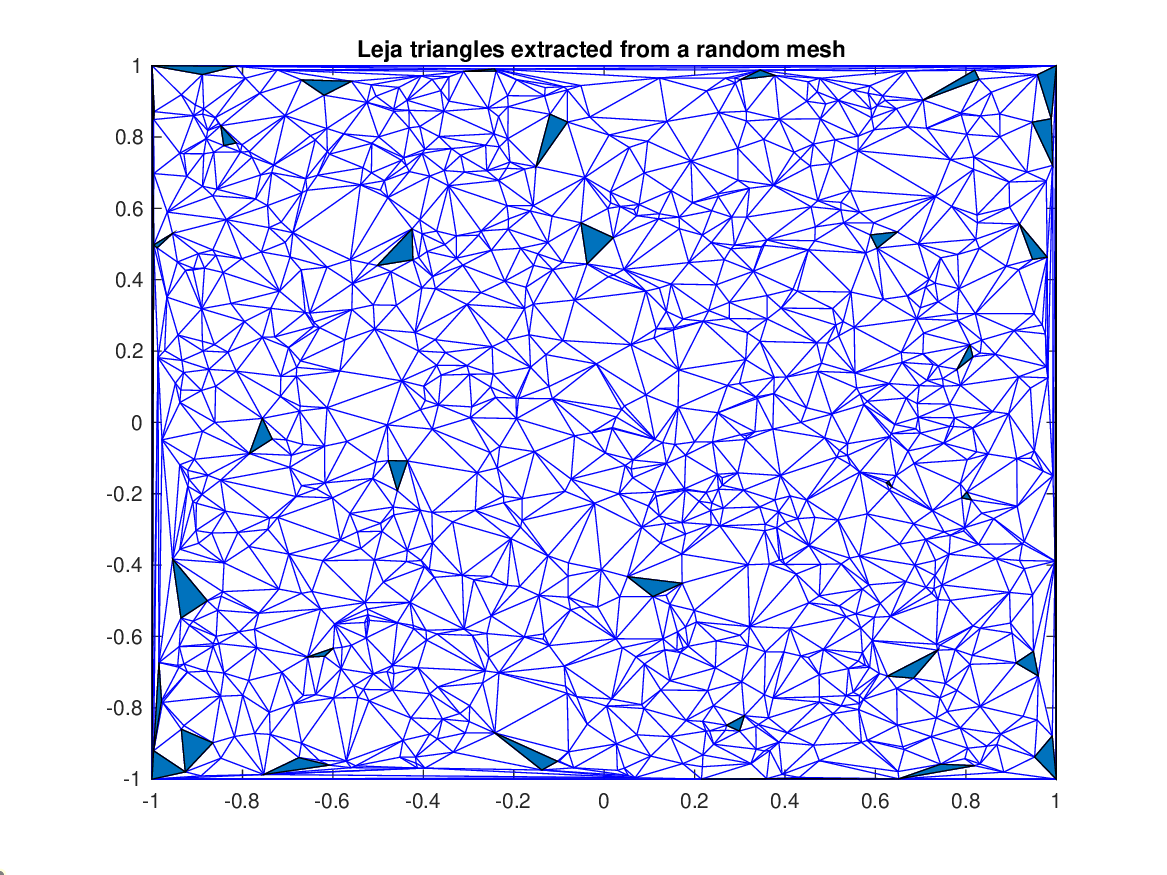}
    \\
    \caption{Extraction of Padua and Leja triangles from a random mesh of about 2000 elements.}
    \label{fig:MeshRandomComparison}
\end{figure}

We finally drop the construction of Padua triangles, and compare the convergence of the histopolants associated with the Fekete and the Leja triangles on random meshes. The respective results are quite consistent with the situation pictured for structured mesh. In this case, the randomness of the mesh is partially compensated by the large number of triangles at play. As a consequence, for $ f_1 $ and $ f_2 $ we observe a slow convergence towards the exact functions, see Figure \ref{fig:MeshRandomConvergence}.

\begin{figure}[H]
    \centering
     \includegraphics[width=0.49\linewidth]{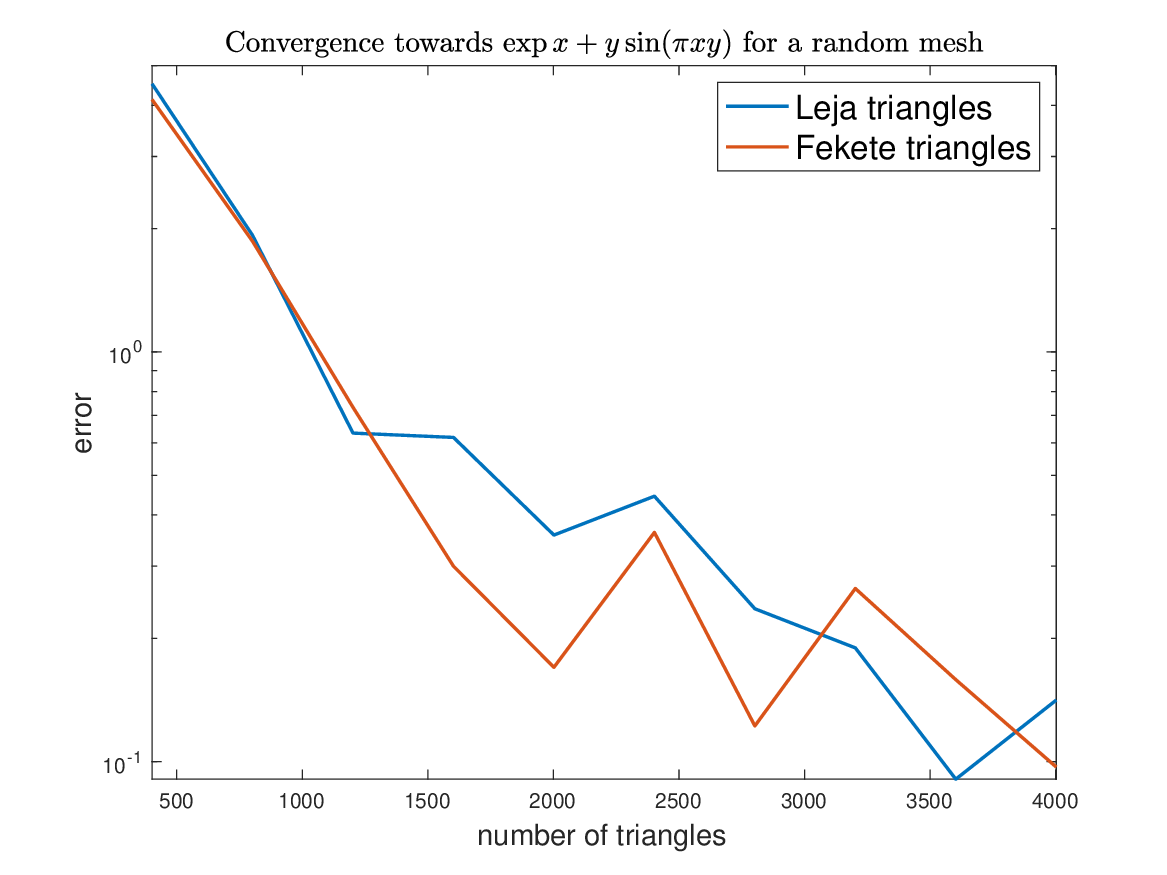} \
     \includegraphics[width=0.49\linewidth]{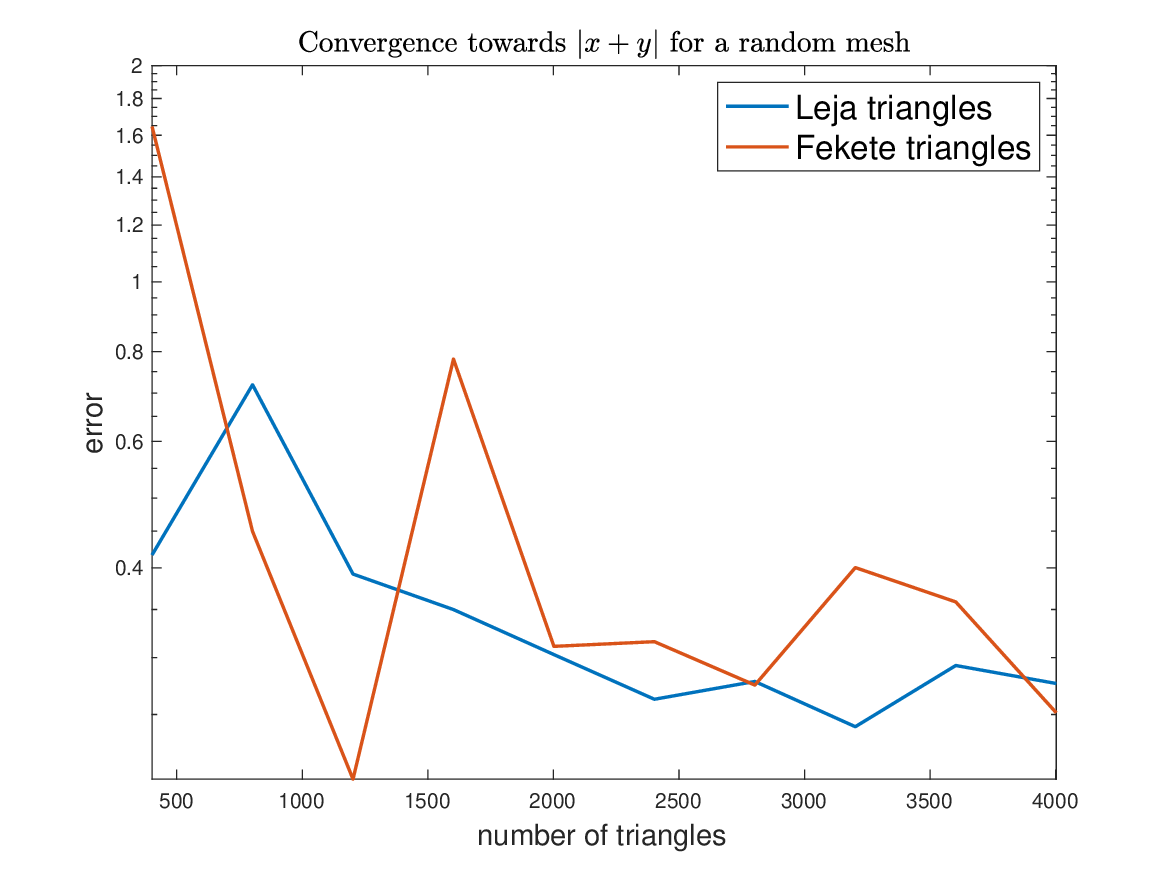}
    \caption{Convergence of the histopolators arising from a random mesh as in Figure \ref{fig:MeshRandomComparison}.}
    \label{fig:MeshRandomConvergence}
\end{figure}

\section*{Declarations}

\vspace{-0.2cm}

\begin{itemize}
    \item Funding: This research was supported by GNCS-INdAM 2024 projects. The first author is funded by IN$\delta$AM and supported by the University of Padova. The third and the fourth authors are funded by the European Union – NextGenerationEU under the National Recovery and Resilience Plan (NRRP), Mission 4 Component 2 Investment 1.1 - Call PRIN 2022 No. 104 of February 2, 2022 of the Italian Ministry of University and Research; Project 2022FHCNY3 (subject area: PE - Physical Sciences and Engineering) \lq\lq Computational mEthods for Medical Imaging (CEMI)\rq\rq. 
    \item Authors' Contributions: All authors contributed equally to this article, and therefore the order of authorship is alphabetical.
    \item Conflicts of Interest: The authors declare no conflicts of interest.
    \item Ethics Approval: Not applicable.
    \item Data Availability: No data were used in this study.
\end{itemize}

\vspace{-0.6cm}


\bibliographystyle{spmpsci}
\bibliography{paper}

\end{document}